\documentclass[12pt]{article}

\usepackage[utf8]{inputenc}
\usepackage[english]{babel}
\usepackage[T1]{fontenc}
\usepackage{amsmath}
\usepackage{amsfonts}
\usepackage{amssymb}
\usepackage{amsthm}
\usepackage{hyperref}
\usepackage{cleveref}
\usepackage{thmtools}
\usepackage{thm-restate}
\usepackage{enumerate}
\usepackage{subcaption}
\usepackage{todonotes}
\usepackage{tabto}
\usepackage{dsfont}
\usepackage{xspace}
\usepackage{relsize}
\usepackage{dsfont}
\usepackage{tikz}
\usepackage{mathtools}
	\usetikzlibrary{backgrounds,calc,decorations.pathmorphing}
	\tikzset{every path/.style={thick}}
	\pgfdeclarelayer{bg}    % declare background layer
	\pgfsetlayers{bg,main}  % set the order of the layers (main is the standard layer)
\usepackage{pgfplots}

\pgfplotsset{compat=newest}
\usetikzlibrary{patterns}
\usepgfplotslibrary{fillbetween}

\usepackage[centering,
			textwidth=15cm,
			textheight=22cm,
			top=3.5cm,
			footskip=40pt,
			marginparwidth=2.9cm
			]{geometry}
			
\theoremstyle{plain}
\newtheorem{thm}{Theorem}[section]
\newtheorem{prop}[thm]{Proposition}
\newtheorem{lem}[thm]{Lemma}
\newtheorem{cor}[thm]{Corollary}
\newtheorem{conj}[thm]{Conjecture}
\newtheorem{obs}[thm]{Observation}

\theoremstyle{definition}
\newtheorem{defi}[thm]{Definition}
\newtheorem{ex}[thm]{Example}

\makeatletter
\newcommand*\rel@kern[1]{\kern#1\dimexpr\macc@kerna}
\newcommand*\widebar[1]{%
	\begingroup
	\def\mathaccent##1##2{%
		\rel@kern{0.8}%
		\overline{\rel@kern{-0.8}\macc@nucleus\rel@kern{0.2}}%
		\rel@kern{-0.2}%
	}%
	\macc@depth\@ne
	\let\math@bgroup\@empty \let\math@egroup\macc@set@skewchar
	\mathsurround\z@ \frozen@everymath{\mathgroup\macc@group\relax}%
	\macc@set@skewchar\relax
	\let\mathaccentV\macc@nested@a
	\macc@nested@a\relax111{#1}%
	\endgroup
}
\makeatother

\newcommand{\R}{\mathbb{R}}
\newcommand{\N}{\mathbb{N}}
\newcommand{\Z}{\mathbb{Z}}
\newcommand{\dotcup}{\hspace{.22em}\ensuremath{\mathaccent\cdot\cup}\hspace{.22em}}
\newcommand{\conv}{\mathrm{conv}}

\newcommand{\Mcut}{\textsc{MultC}}
\newcommand{\PMcut}{\Mcut^\square}
\newcommand{\MinMcut}{\textsc{MultiCut}\xspace}
\newcommand{\supp}{\mathrm{supp}}
\newcommand{\mypar}[1]{\paragraph{#1.}}
\newcommand{\qq}[1]{``#1''}%quotes

\newcommand{\V}{{\widebar{V}}}
\newcommand{\E}{{\widebar{E}}}
\newcommand{\G}{{\widebar{G}}}
\renewcommand{\a}{{\widebar{a}}}
\newcommand{\wS}{{\widebar{S}}}

\newcommand{\e}{{e^*}}
\newcommand{\bnull}{{\bf 0}}

\newcommand{\T}{\mathsf{T}}

\newcommand{\nice}{shared\xspace}

\providecommand{\keywords}[1]{\textbf{\textit{keywords---}} #1}

\title{On the Dominant of the Multicut Polytope}
\author{Markus Chimani, Martina Juhnke-Kubitzke, Alexander Nover}
\date{\smaller School of Mathematics/Computer Science, Uni Osnabrück, Germany\\
\{markus.chimani,juhnke-kubitzke,alexander.nover\}@uni-osnabrueck.de}
\begin{document}
\maketitle
\begin{abstract}
	Given a graph $G=(V,E)$ and a set $S \subseteq \binom{V}{2}$ of terminal pairs, the minimum multicut problem asks for a minimum edge set $\delta \subseteq E$ such that there is no $s$-$t$-path in $G -\delta$ for any $\{s,t\}\in S$.
	For $|S|=1$ this is the well known $s$-$t$-cut problem, but in general the minimum multicut problem is NP-complete, even if the input graph is a tree.
	The multicut polytope $\PMcut(G,S)$ is the convex hull of all multicuts in $G$; the multicut dominant is given by
	$\Mcut(G,S)=\PMcut(G,S)+\R^E$. The latter is the relevant object for the minimization problem. While polyhedra associated to several cut problems have been studied intensively there is only little knowledge for multicut.
	
	We investigate properties of the multicut dominant and in particular derive results on liftings of facet-defining inequalities. This yields a classification of all facet-defining path- and edge inequalities. Moreover, we investigate the effect of graph operations such as node splitting, edge subdivisions, and edge contractions on the multicut-dominant and its facet-defining inequalities. In addition, we introduce facet-defining inequalities supported on stars, trees, and cycles and show that the former two can be separated in polynomial time when the input graph is a tree.
\end{abstract}

\section{Introduction}
	Cut problems on graphs are well known in combinatorial optimization.
	Probably best known are the minimum cut problem, which is polynomial time solvable \cite{ford_fulkerson_1956}, and the maximum cut problem, which is one of Karp's original 21 NP-complete problems \cite{karp}.
	
	A prominent generalization of the minimum $s$-$t$-cut problem is the \emph{minimum multicut problem} $\MinMcut$:
	Given a graph $G$ and a set $S\subseteq \binom{V(G)}{2}$ of terminal pairs, a \emph{multicut} is an edge set $\delta \subseteq E(G)$ such that for each pair $\{s,t\}\in S$ there is no $s$-$t$-path in $G -\delta$.
	Given non-negative edge weights $c_e$, $\MinMcut$ asks for a multicut $\delta$ minimizing $\sum_{e \in \delta} c_e$.
	
	If $|S|$ is fixed, \MinMcut is solvable in polynomial time for $|S|=1,2$ \cite{2-Terminals} but NP-complete for $|S| \geq 3$ \cite{Np-Hardness_3Terminals}.
	It remains NP-complete even when the input graph is restricted to trees of height~1, i.e., stars~\cite{NP-hard_trees}.
	Approximation algorithms for \MinMcut have been intensively studied, see e.g. \cite{approx_k-multicut, NP-hard_trees}.
	
	In this work we consider \MinMcut from a polyhedral point of view. 
%	To each multicut $\delta \subseteq E(G)$ we associate its incidence vector $x^\delta \in \R^{E(G)}$ given by
%	$$x^\delta_e=\begin{cases}
%		1, &\text{if } e \in \delta,\\
%		0, &\text{else.}
%	\end{cases}$$
	The \emph{multicut polytope} $\PMcut(G,S)$ is the convex hull of all incindence vectors of multicuts; the \emph{multicut dominant} is given by $\Mcut(G,S)=\PMcut(G,S)+\R^{E(G)}_{\geq 0}$. Since minimizing a non-negative objective function on $\PMcut(G,S)$ and $\Mcut(G,S)$ yields the same result, the latter is the relevant polyhedron for the considered optimization problem.
	
	A large body of the research on the cut polytope -- the polytope associated to the maximum cut problem-- consists of investigating facet-defining inequalities for certain graph classes \cite{DL1,DL2,GeometryOfCutsAndMetrics,POLJAK1992379,OnTheCutPolytope}.
	This is driven by the fact that these give rise to valid inequalities for the cut polytope of every graph containingsuch a graph as a subgraph.
	Moreover, there is an extensive study of the effect of graph operations such as node splitting and edge subdivisions on the cut polytope and its facet-defining inequalities \cite{OnTheCutPolytope,GeometryOfCutsAndMetrics}.
	
	Contrary to the cut polytope, there is only little knowledge on the multicut polytope and its dominant.
	
	For $|S|=1$, the multicut dominant was studied in \cite{s-t-cut_dominant}. Besides a characterization of vertices and adjacencies in the polyhedron, it was shown that in this case $\Mcut(G,\{\{s,t\}\})$ is completely described by \emph{edge-} and \emph{path inequalities} (see \Cref{prop:facets_s-t-cut} for details).
	Moreover, it was shown that each of these inequalities defines a facet.
	
	Clearly this generalizes to a relaxation of $\Mcut(G,S)$ for $|S|\geq 2$ by having path inequalities for each pair $\{s,t\}\in S$.
	In \cite{approx_k-multicut} it was shown that when the input graph is a tree and for each $\{s,t\}\in S$ one of both nodes is a descendant of the other this relaxation coincides with $\Mcut(G,S)$. Nevertheless, this does not hold in general.
	Already for $G=K_{1,3}$ with $S=\{\{v,w\}:v,w \text{ are leaves in } G\}$ the polyhedron defined by all edge- and path inequalities admits a fractional vertex by setting all edge variables to $0.5$.
	
	\mypar{Our Contribution}
		After recalling definitions and formally introducing the multicut dominant as the main object of our studies in \Cref{sec: preliminaries}, we start \Cref{sec:basic_properties} by investigating basic properties of $\Mcut(G,S)$.
		Moreover, we present results on liftings and projections of these polyhedra.
		The lifting results for the multicut dominant are stronger than those known for cut polytopes in the sense that lifting of inequalities does not only preserve validity of the inequalities but also preserves being facet-defining.
		This results in a characterization of facet-defining edge- and path inequalities.
		Then, we investigate the effect of graph operations such as node splittings and edge subdivisions on the multicut dominant and its facets in \Cref{sec:constructing_facets}.
		In \Cref{sec:star-ineq} we investigate facets supported on stars.
		In \Cref{sec:tree-ineq} we generalize these facet-defining inequalities  to facets on trees.
		Both classes can be separated in polynomial time when the input graph is a tree.
		Finally, in \Cref{sec:cycle-ineq} we introduce facet-defining inequalities supported on cycles.
		
		\mypar{Related Cut-Gerneralizations}
		There are multiple way to generalize cuts to different problems under the same (or similar) name in literature.
		
		In \cite{small_multicut_polytopes,k-cut-polytope} multiple polytopes associated to cut problems are studied.
		There, the \emph{k-cuts} are called multicuts as well;
		we give give their to distinguish those from our notion of multicuts:
		Given a graph $G=(V,E)$ and a partition $V= S_1 \dotcup \dots \dotcup S_k$ a \emph{$k$-cut} in $G$ is the set of all edges between a node in $S_i$ and a node in $S_j$ for some $1 \leq i < j \leq k$.

		Using this notion of multicuts, in \cite{lifted_multicut,lifted_multicut_polytope} the \emph{lifted multicut problem} was studied:
		Given a graph $G'$, a subgraph $G \subseteq G'$, and a multicut $\delta \subseteq E(G)$, the lifted multicut problem asks for a minimum multicut in $\widebar \delta \subseteq E(G')$ with $\widebar \delta \cap E(G)=\delta$.
		The polytope associated to this problem is called \emph{lifted multicut polytope}.
		In \cite{lifted_multicut_polytope}, the lifted multicut polytope for $G$ being a tree or a path was studied.

\section{Preliminaries}\label{sec: preliminaries}
%	In this section we provide some basic background on graphs and polyhedra. Then, we introduce the multicut dominant.  For notation and results related to graphs we refer to \cite{diestel}, for those related to polyhedra to \cite{BrunsGubeladze,ziegler}.
%	
%	\mypar{Graphs.}
	We only consider undirected graphs. A graph is \emph{simple}, if it has neither parallel edges nor self-loops.
	Unless specified otherwise, we only consider simple graphs in the following. 
	Given a graph $G=(V,E)$ we may also write $V(G)$ and $E(G)$ for its set of nodes $V$ and its set of edges $E$, respectively.
	For $v,w \in V(G)$, we let $vw=\{v,w\}$ be the edge between $v$ and $w$.
%	
%	A \emph{path} of length $k$ is a sequence of edges $e_1, \dots, e_k$ with $e_i =v_{i-1} v_i$ such that $v_i \neq v_j$ for $0 \leq i < j \leq k$. 
%	For nodes $s$ and $t$ an $s$-$t$-path is a path with $v_0=s$ and $v_k=t$. 
%	Such a sequence but with $v_0 = v_k$ is a \emph{cycle} of length $k$.
%	We use $C_n$ to denote the cycle of length~$n$ and $K_{m,n}$ to denote the complete bipartite graph on $m$ and $n$ nodes per partition set.
%	
%	A graph $H$ is a \emph{subgraph} of $G$, denoted by $H\subseteq G$, if (after possibly renaming) $V(H)\subseteq V(G)$ and $E(H) \subseteq E(G)$.
%	Given a subset $W \subseteq V$, the subgraph \emph{induced} by $W$ is the graph $G[W]=(W, \{uv \in E: u,v \in W \})$.
%	For nodes $u,v \in V(G)$ the \emph{distance} $d(u,v)$ between $u$ and $v$ is the length of a shortest path in $G$ from $u$ to $v$. for a node $v \in V(G)$ and a subgraph $H \subseteq G$, the \emph{distance} between $H$ and $v$ is given by $d(H,u)=\min_{u \in H}d(u,v)$.
%	
%	We denote the graph obtained from $G$ by deleting vertices $v_1,\dots, v_k$ (resp. edges $e_1,\dots, e_k$) by $G-\{v_1,\dots,v_k\}$ (resp. $G-\{e_1,\dots,e_k\}$). If we remove a single node $v$ (resp. edge $e$), we might just write $G-v$ (resp. $G-e$).
%	The graph $G/e$ is obtained from $G$ by \emph{contracting} the edge $e=vw$, i.e., the nodes $v$ and $w$ are identified, the arising self-loop is deleted and parallel edges are merged.
%	$G$ is \emph{$k$-connected} if $|V(G)| \geq k+1$ and for each pair of nodes $v,w \in V(G)$ there exist $k$ internally node-disjoint paths from $v$ to $w$.
	
%	\mypar{Polyhedra.}
	A \emph{polytope} (see, e.g., \cite{BrunsGubeladze,ziegler} for details) is the convex hull of finitely many points in $\R^d$.
	An (unbounded) \emph{polyhedron} is the Minkowski sum of a polytope and a cone generated by finitely many points. 
	A polyhedron is a polytope if and only if it is bounded.
	
	In the following let $\mathcal{P}$ be a polyhedron. 
	The \emph{dimension} $\dim \mathcal{P}$ is the dimension of its affine hull. 
	A linear inequality $a^\mathsf{T}x \geq b$ where $a \in \R^d$ and $b \in \R$ is a \emph{valid inequality} for $\mathcal{P}$ if it is satisfied by all points $x \in \mathcal{P}$.  
	It is \emph{tight} if there is some $p \in \mathcal{P}$ with $a^\mathsf{T}p=b$. We use the shorthand $\{a^\mathsf{T}x\geq b\}$ for $\{x \in \R^E : a^\mathsf{T}x \geq b\}$ and its analogon for equalities.
%	A (proper) \emph{face} of $P$ is a (non-empty) set of the form $P \cap \{a^\mathsf{T}x = b\}$ for some valid inequality $a^\mathsf{T} x \geq b$ with $a\neq \bnull$.
%	Faces of dimension~$0$ and $\dim(P)-1$ are \emph{vertices} and \emph{facets}, respectively. 
%	For each face $F \subseteq P$ there is a facet $F' \subseteq P$ \emph{dominating} it, i.e., $F \subseteq F'$.
%	
%	A tight inequality $a^\mathsf{T}x \geq b$ is \emph{facet-defining} if $P\cap\{a^\mathsf{T}x = b\}$ is a facet of $P$.
%	Each polyhedron can be represented as the intersection of finitely many closed half-spaces, i.e., $P$ admits a \emph{linear description} $P=\{x \in \R^d: Ax \geq {b}\}$ for some matrix $A\in \R^{m \times d}$ and some vector ${b}\in \R^m$.
%	This is given, e.g., by taking the system of all facet-defining inequalities.
	
\mypar{The Multicut Dominant}
	For $k \in \mathbb{N}$, let $[k]=\{1, \dots , k\}$.
	Given a graph $G$ and a set $S=\{\{s_1,t_1\},\dots, \{s_k,t_k\} \}\subseteq \binom{V(G)}{2}$ of \emph{terminal pairs} a \emph{multicut} is a set $\delta \subseteq E(G)$ such that for each $i \in [k]$ the nodes $s_i$ and $t_i$ are in different components of $G-\delta$.
	When the terminal set is in doubt, we may call $\delta$ an \emph{S-multicut}.
	A node $v \in V(G)$ is called a \emph{terminal} if there exists some $w \in V(G)$ such that $\{v,w\}\in S$. A multicut is \emph{minimal} if it is minimal with respect to inclusion, it is a \emph{minimum multicut} if it has minimal total weight (with respect to given edge weights).

	To each edge set $F \subseteq E$ we associate its incidence vector $x^F \in \R^E$ given by
	$$x^F_e= \begin{cases}	1, & \text{ if } e \in F,\\
		0, & \text{ else.}			\end{cases}$$
	We define the \emph{multicut polytope} of $G$ as
	$$\PMcut(G,S)= \conv \left(\left\{x^\delta:\ \delta \text{ is a multicut in $G$ with respect to $S$} \right\} \right) $$
	and the \emph{multicut dominant} of $G$ as
	$$\Mcut(G,S)=\PMcut(G,S)+ \R_{\geq 0}^{E(G)}.$$
	
	Given a valid inequality $a^\mathsf{T}x \geq b$ of $\Mcut(G,S)$ its \emph{support graph} $\supp(a)\subseteq G$ is the subgraph of $G$ induced by the edge set $\{e \in E(G): a_e \neq 0\}$.
	If $\supp(a)$ is a single edge, the inequality is called an \emph{edge inequality}.

\section{Basic Properties}\label{sec:basic_properties}
	We start by investigating basic properties of the multicut dominant and its facet-defining inequalities.
	Afterwards, we study the effect of edge additions, deletions, and contractions on the multicut dominant.
	In particular, this leads to a classification of all facet-defining path- and edge inequalities.
	
	The following observation is a direct consequence of the construction of the multicut dominant:
	\begin{obs} Let $G=(V,E)$ be a graph and $S \subseteq \binom{V}{2}$ be a set of terminal pairs.
		\begin{itemize}
			\item The vertices of $\Mcut(G,S)$ are precisely the incidence vectors of the (inclusion wise) minimal multicuts in $G$.
			\item We have $\dim \Mcut(G,S)=|E|$.
			\item Let $a^\mathsf{T} x \geq b$ be facet-defining for $\Mcut(G,S)$. Since $\Mcut(G,S)$ is the Minkowski sum of a polytope and $\R^E_{\geq 0}$, each inner normal of a facet of $\Mcut(G,S)$ is contained in $\R_{\geq 0}^E$, i.e. we have $a_e \geq 0$ for each $e \in E$.
			\item Let $W \supseteq V$, $\widebar G=(W,E)$, and $\widebar S \subseteq \binom{W}{2}$  be a set of terminal pairs such that $S=\widebar S \cap \binom{V}{2}$. Then, we have $\Mcut(G,S)=\Mcut(\widebar G,S)$.
		\end{itemize}
	\end{obs}
	
		We can consider the support graph of facet-defining inequalities:
	\begin{lem}\label{lem:nonzero_coeff_support}
		 Let $G=(V,E)$ be a graph, $S \subseteq \binom{V}{2}$ be a set of terminal pairs, $a^\T x \geq b$ be facet-defining for $\Mcut(G,S)$, and $f \in E$. Assume that $a_f \neq 0$ and $\{a^\T x \geq b\} \neq \{x_f \geq 0\}$. Then, there exists some multicut $\delta$ with $f \in \delta$ and $a^\T x^\delta =b$.
	\end{lem}
	\begin{proof}
		Assume there is no such $\delta$ and let $\lambda >0$. The inequality $a^\T x + \lambda x_f \geq b$ is valid for $\Mcut(G,S)$. By assumption, each multicut $\delta$ satisfying $a^\T x =b$ satisfies $a^\T x + \lambda x_f=b$. Hence, both inequalities define the same face of $\Mcut(G,S)$ contradicting the assumption that $a^\T x \geq b$ is facet-defining.
	\end{proof}
	
	\begin{thm}
		Let $G=(V,E)$ be a graph, $S \subseteq \binom{V}{2}$ be a set of terminal pairs, and $a^\mathsf{T} x \geq b$ be facet-defining for $\Mcut(G,S)$ such that  $\{a^\T x \geq b\} \neq \{x_{e'} \geq 0\}$ for all $e' \in E(G)$. Then, each edge $e \in E(\supp(a))$ lies on an $s$-$t$-path in $\supp(a)$ for some $\{s,t\} \in S$.
		In particular, each leaf of $\supp(a)$ is a terminal.
	\end{thm}
	\begin{proof}
		Assume there is some $e \in \supp(a)$ such that $e$ does not lie on any $s$-$t$-path.
		By \Cref{lem:nonzero_coeff_support} there is some multicut $\delta$ with $e \in \delta$ and $a^\mathsf{T} x^\delta =b$. Since $e$ is not contained in any $s$-$t$-path for any $\{s,t\} \in S$, also $\delta \setminus \{e\}$ is a multicut. But, since $e \in \supp(a)$, we have $a_e >0$ and thus, $a^\mathsf{T}x^{\delta \setminus \{e\}} < a^\mathsf{T}x^\delta =b$ contradicting $a^\mathsf{T}x \geq b$ being valid for $\Mcut(G,S)$.
	\end{proof}
	
	Next, we investigate coefficients of facet-defining inequalities of the multicut dominant along \emph{induced paths}, i.e., paths in $G$ in which each internal node has degree~$2$ in $G$.
	\begin{thm}
		 Let $G=(V,E)$ be a graph, $S \subseteq \binom{V}{2}$ be a set of terminal pairs, and $a^\T x \geq b$ be facet-defining for $\Mcut(G,S)$ with $\{a^\T x \geq b\} \neq \{x_{e'} \geq 0\}$ for all $e' \in E$. Furthermore let $P \subseteq G$ be an induced path such that no internal node of $P$ is a terminal.
		Then, $a_e =a_f$ for all $e,f \in P$.
	\end{thm}
	\begin{proof}
		Assuming the contrary, let $M=\min_{e \in E(P)}a_e$, $e\in E(G)$ with $a_e=M$, and define $c \in \R^E$ by
		$$c_e=\begin{cases}
			a_e,	&\text{for } e \notin E(P),\\
			M,		&\text{for } e \in E(P).
		\end{cases}$$
		First, we show that $c^\mathsf{T} x \geq b$ is valid for $\Mcut(G,S)$. 
		Let $\delta$ be a multicut in $G$.
		Since $P$ is induced and contains no terminals as inner vertices, also $\delta'=\left(\delta \setminus P\right) \cup \{e\}$ is a multicut  in $G$. 
		Clearly, we have $c^\T x^\delta \geq c^\T x^{\delta'} =a^\mathsf{T}x^{\delta'} \geq b$. 
		
		Since $a^\T x \geq b$ is a sum of $c^\T x \geq b$ and edge inequalities $x_e \geq 0$, this contradicts the assumption that $a^\T x \geq b$ is facet-defining.
	\end{proof}
	We want to point out that both assumptions on the path in the previous theorem are necessary. \Cref{thm: complete_tree_ineq} will provide facet-defining inequalities having non-induced paths with different coefficients in the support graph.
	The facets in \Cref{thm: WagnerIneq} contain induced paths with internal terminals in their support graphs and have different coefficients attached to edges in such paths.
	
	Next, we give a complete characterization of the boundedness of facets via their support graph.
	\begin{thm}\label{thm: boundedness}
		 Let $G=(V,E)$ be a graph, $S \subseteq \binom{V}{2}$ be a set of terminal pairs, and $a^\T x \geq b$ be facet-defining for $\Mcut(G,S)$. Then the facet $\{a^\T x = b\}\cap \Mcut(G,S)$ is bounded if and only if $\supp(a)=G$.
	\end{thm}
	\begin{proof}
		Since $\Mcut(G,S)=\PMcut(G,S)+ \R^E_{\geq 0}$, each ray in $\Mcut(G,S)$ is of the form $\{y+ \lambda z: \lambda \in \R_{\geq 0}\}$ with $y \in \Mcut(G,S)$ and $z \in \R^E_{\geq 0}\setminus\{\bnull\}$.
		
		If there is some edge $e \in E\setminus E(\supp(a))$, we have $a^\T x = a^\T (x+x^{\{e\}})$ yielding that  $\{a^\T x = b\}\cap \Mcut(G,S)$ is unbounded.
		
		Now assume that $\supp(a)=G$ and there is a ray $\{y+ \lambda z: \lambda \in \R_{\geq 0}\} \subseteq\{a^\T x = b\}\cap \Mcut(G,S)$. Then, we have
		$0=b-b=a^T(y+ \lambda z)-a^\T y=\lambda a^\T z$ for each $\lambda \in \R_{\geq 0}$. Thus, we have $a^T z=0$ and since $a_e > 0$ for all $e \in E$ this contradicts $z \in \R^E_{\geq 0}\setminus\{\bnull\}$.
	\end{proof}

	The graph $G/e$ is obtained from $G$ by \emph{contracting} the edge $e=vw$, i.e., the nodes $v$ and $w$ are identified, the arising self-loop is deleted and parallel edges are merged.
	Considering the contraction of an edge $e$, there is a one-to-one correspondence between multicuts in $G/e$ and multicuts $\delta$ in $G$ with $e \notin \delta$. This, directly yields the following observation:
	\begin{obs}\label{obs: contraction}
		Let $G=(V,E)$ be a graph, $S \subseteq \binom{V}{2}$ be a set of terminal pairs, and $e=vw \in E$ such that $\{v,w\}\notin S$. Then $\Mcut(G/e,S)=\Mcut(G,S)\cap \{x_e=0\}$.
	\end{obs}
	
	Next, we consider the deletion and addition of edges.
	\begin{thm}\label{thm: projection}
		Let $G=(V,E)$ be a graph, $S \subseteq \binom{V}{2}$ be a set of terminal pairs, $H \subseteq G$, and $S'=S \cap \binom{V(H)}{2}$. 
		Then, the following hold:
		\begin{enumerate}[(i)]
			\item For $e \in E$, $\Mcut(G-e,S)=\pi(\Mcut(G,S))$ where $\pi\colon \R^E \to \R^{E \setminus \{e\}}$ is the orthogonal projection.
			\item If $\sum_{e \in E(H)}a_ex_e \geq b$ is valid for $\Mcut(H,S')$, it is also valid for $\Mcut(G,S)$.
			\item If $a^\mathsf{T}x \geq b$ is facet-defining for $\Mcut(G,S)$ and $\supp(a) \subseteq H$, the inequality $\sum_{e \in E(H)}a_ex_e \geq b$ is facet-defining for $\Mcut(H,S')$.
			\item If $a^\T x \geq b$ is facet-defining for $\Mcut(H,S')$, then $\sum_{e \in E(H)}a_e x \geq b$ is facet-defining for $\Mcut(G,S)$.
		\end{enumerate}
	\end{thm}
	\begin{proof}
		We start by proving {\it (i)}. Then {\it (ii)} and {\it (iii)} follow immediately.
		Observe that for each multicut $\delta$ in $G$ the set $\delta \setminus\{e\}$ is a multicut in $G-e$. On the other hand, if $\delta'$ is a multicut in $G-e$, then $\delta' \cup\{e\}$ is a multicut in $G$. This directly yields $\Mcut(G-e,S)=\pi(\Mcut(G,S))$.
		
		Now, we prove statement {\it(iv)}. If $W \neq V$, we consider the graph $\widebar{H}=(V,F)$. Since $S'$-multicuts in $H$ and $S$-multicuts in $\widebar H$ coincide, we have $\Mcut(H,S')=\Mcut(\widebar{H},S)$. Thus, we may assume $W=V$.
		We prove the statement for the case $H^*=(V,E \setminus \{e^*\})$ for some $e^* \in E$.
		Then, the claim follows by adding edges in $E \setminus F$ one by one to $H$.
		
		For each multicut $\delta$ in $H^*$, the set $\widebar{\delta}=\delta \cup\{e^*\}$ is a multicut in $G$.
		Thus, ${x^{\widebar{\delta}}, x^{\widebar{\delta}}+x^{\{e^*\}} \in \Mcut(G,S)}$.
		Since $a ^\mathsf{T} x \geq b$ is facet-defining for $\Mcut(G,S)$, lifting all multicuts satisfying $a^\T x = b$ in this fashion into the hyperplanes ${\{x_{e^*}=1\}}$ and $\{x_{e^*}=2\}$ yields that
		the inequality $\sum_{e\in E\setminus\{e^*\}}a_ex_e \geq b$ defines a face of dimension at least $|E|$ of $\Mcut(G,S)$. Since $a \neq \bnull$, this yields that the inequality is facet-defining.
	\end{proof}
	
	Note that the facet-defining inequalities of the multicut dominant can be split into two sets: those that are also facet-defining for the multicut polytope and those that are not.

	\begin{defi}[shared facets]
			Let $G=(V,E)$ be a graph, $S\subseteq \binom{V}{2}$ be a set of terminal pairs.
			Given a facet-defining inequality $a^\T x \geq b$ of $\Mcut(G,S)$, the defined facet is \emph{\nice} if $a^\T x \geq b$ is also facet-defining for $\PMcut(G,S)$.
	\end{defi}

	Clearly, each facet-defining inequality $a^\mathsf{T} x \geq b$ of $\Mcut(G,S)$ with $\supp(a)=G$ is \nice by \Cref{thm: boundedness}.
	Moreover, considering the proof of \Cref{thm: projection}, one can see that \nice facets remain \nice under removal of edges.
	Unfortunately, as we also see in that proof this does not hold for lifting in general.
	However, with some additional restrictions we can lift facet-defining inequalities while making sure that the property of being \nice is preserved.
	\begin{lem}\label{lem:0lifting_addEdge_nicefacet}
		Let $G=(V,E)$ be a graph, $S \subseteq \binom{V}{2}$ be a set of terminal pairs, $a^\T x \geq b $ define a \nice facet of $\Mcut(G,S)$, $v,w \in V$ such that $e^*=vw \notin E$, and $\G=(V,E\cup\{e^*\})$.
		Assume there is a multicut $\delta^*$ in $\G$ such that $e^* \notin \delta^*$ and $\sum_{e \in E}a_e x^{\delta}_e =b$. Then
		$\sum_{e \in E}a_ex_e \geq b$ defines a \nice facet of $\Mcut(\G,S)$.
	\end{lem}
	\begin{proof}
		It follows directly from \Cref{thm: projection} (iv) that the inequality is facet-defining. Thus, it is only left to show that that the defined facet is indeed \nice, i.e., that there are $|\E|$ affinely independent incidence vectors of multicuts contained in this facet.
		
		To this end, let $m=|E|$.
		Since $a^\T x \geq b$ is facet-defining for $\Mcut(G,S)$, there are multicuts $\delta_1, \dots, \delta_m$ in $G$ such that $a^\T x^{\delta_i}= b$ for each $i \in [m]$ and $x^{\delta_1}, \dots, x^{\delta_m}$ are affinely independent. Now let $\widebar{\delta}_i= \delta_i \cup \{e^*\}$. Then, $\widebar{\delta}_i$ is a multicut in $\G$ with $\sum_{e \in E} a_e x^{\widebar{\delta}_i}_e =b$. Since $x^{\widebar{\delta}_1}, \dots, x^{\widebar{\delta}_m}$ are affinely independent and all contained in the hyperplane $\{x_{e^*}=1\}$, the vectors $x^{\delta^*},x^{\widebar{\delta}_1}, \dots, x^{\widebar{\delta}_m}$ are affinely independent and satisfy $\sum_{e \in E}a_e x_e=b$. 
	\end{proof}
	
	\begin{thm}\label{thm:lifting_inducedSubgraph}
		Let $G=(V,E)$ be a graph, $S \subseteq \binom{V(G)}{2}$, $H=(W,F)$ be an induced subgraph of $G$, and $S'=S \cap \binom{W}{2}$. 
		Assume that  for each $vw \in E \setminus F$ we have $\{v,w\}\notin S$ and there is no $\{s,t\}\in S$ with $s \in W$ and $t$ adjacent to $W$.
		Furthermore let $a^\T x \geq b$ define a \nice facet $\Mcut(H,S')$.
		Then, $\sum_{e \in F}a_e x_e \geq b$ defines a \nice facet of $\Mcut(G,S)$.
	\end{thm}
	\begin{proof}
		Let $G'=(V,F)$. Clearly, each $S'$-multicut in $H$ is an $S$-multicut in $G'$ and we thus have $\Mcut(H,S')=\Mcut(G',S)$.
		
		In the following we add the edges in $E \setminus F$ one by one to $G'$ and utilize \Cref{lem:0lifting_addEdge_nicefacet} to lift the facet-defining inequality under consideration.
		To this end we construct for each each $e \in E\setminus F$ an $S$-multicut $\widebar{\delta}_e \subseteq E$ in $G$  such that $e \notin \widebar{\delta}_e$ and $\sum_{f \in E}a_f x^{\widebar{\delta}_e} =b$. Since this multicut induces an according multicut in each step, this yields the claim.
		
		Now, let $e=vw \in E\setminus F$, $\delta \subseteq F$ be an $S$-multicut in $H$ (and thus in $G'$) with $a^\T x^\delta =b$, and $\widebar{\delta}_e=(\delta \cup (E\setminus F)) \setminus \{e\}$.
		Since $H$ is induced and there is no  $\{s,t\}\in S$ with $s \in W$ and $t$ adjacent to $W$,
		for each $\{s,t\} \in S$ there is no $s$-$t$-path in
		$$G-\widebar{\delta}_e=\left( V \setminus (W \cup \{v,w\}),\emptyset \right)\dotcup \left( (H -\delta) \cup (\{v,w\},\{e\})\right).$$
		Thus, $\widebar{\delta}_e$ is an $S$-multicut in $G$.
		Furthermore, we have $\sum_{f \in F} a_fx^{\widebar{\delta}_e}_f =b$.
	\end{proof}
	
	Finally, we investigate edge- and path inequalities for the multicut dominant. To this end, we recall the facet description for the s-t-cut dominant:
	\begin{prop}\cite[Section 2]{s-t-cut_dominant}\label{prop:facets_s-t-cut}
		Let $G=(V,E)$ and $s,t \in V$. Then, the \emph{$s$-$t$-cut dominant} $\Mcut(G,\{\{s,t\}\})$ is completely defined by the inequalities
		\begin{alignat*}{2}
			x_e						&\geq 
			\begin{cases}
				1, &\text{if $e=st$,}\\
				0, & \text{otherwise.}
			\end{cases}
			\qquad	&&\text{for all $e \in E$ },			\\
			\sum_{e \in E(P)} x_e	& \geq 1, 			&&\text{for all $s$-$t$-paths } P.
		\end{alignat*}
		In particular, each of these inequalities defines a facet of  $\Mcut(G,\{\{s,t\}\})$.
		\end{prop}
	Together with \Cref{thm: projection} (iv) this yields the following:
	\begin{cor}\label{cor:edge_and_path-ineq}
		Let $G=(V,E)$, $S$ be a set of terminal pairs and $\{s,t\}\in S$. Then, the following hold:
		\begin{enumerate}[(i)]
			\item for each $vw \in E$ the inequality 
				$$x_{vw} \geq \begin{cases}
				1,	&\text{if }\{v,w\}\in S,\\
				0,	&\text{otherwise}
				\end{cases}$$
				is facet-defining for $\Mcut(G,S)$.
			\item For each $s$-$t$-path $P\subseteq G$ such that there does not exist $\{s',t'\}\in S$ with $s',t' \in V(P)$ and $\{s,t\}\neq \{s',t'\}$, the inequality 
			 $\sum_{e \in E(P)}x_e \geq 1$ is facet-defining for $\Mcut(G,S)$
		\end{enumerate}
	\end{cor}

	Let $a^\T x \geq b$ be facet-defining for some $\Mcut(G,S)$. In general, this inequality is not facet-defining for $\Mcut(G,S')$ with $S' \supset S$.
	\begin{obs}
		Consider a path $P=([n],\{\{i,i+1\}: 1 \leq i < n\})$, $S'=\{(1,n)\}$, and $S=\{\{1,n\},\{i,j\}\}$ for some $\{i,j\}\in \binom{[n]}{2}\setminus \{\{1,n\}\}$.
		Then $\sum_{e \in E(P)} \geq 1$ is facet-defining for $\Mcut(P,S')$ but not for $\Mcut(P,S)$. Moreover, this carries over to arbitrary $G$ with $P \subseteq G$.
	\end{obs}
	
	We close the discussion of path- and edge inequalities by investigating the relation of the polyhedron defined by these inequalities and the multicut dominant:
	\begin{lem}
		Let $G=(V,G)$ be a graph and $S \subseteq \binom{V}{2}$ be a set of terminal pairs. Let $\mathcal P$ be the polyhedron defined by the inequalities
		\begin{alignat*}{2}
			x_{vw}					& \geq
			\begin{cases}
				1,	&\text{if $\{v,w\} \in S$,}\\
				0,	&\text{else}
			\end{cases}
			\qquad&&\text{for all $vw \in E$},			\\
			\sum_{e \in E(P)} x_e	& \geq 1 			&&\text{for all $s$-$t$-paths $P$ with $\{s,t\}\in S$.}
		\end{alignat*}
		Then, the integer points in $\mathcal P$ are precisely those in $\Mcut(G,S)$, i.e., $\mathcal P \cap \Z^n=\Mcut(G,S)\cap \Z^n$.
	\end{lem}
	\begin{proof}
		Since each $S$-multicut in $G$ is an $s$-$t$-cut in $G$ for all $\{s,t\} \in S$, \Cref{prop:facets_s-t-cut} yields $\Mcut(G,S)\subseteq \mathcal P$.
		Now, let $x \in \mathcal P \cap \Z^E$ and define $x' \in\{0,1\}^E$ by setting $x'_e=1$ if $x_e \neq 0$ and $x'_e=0$ otherwise.
		Since $x \in \mathcal P$ we have $x \in \Mcut(G,\{s,t\})$ and thus, $x'\in \Mcut(G,\{s,t\})$ for all $\{s,t\} \in S$.
		Hence, $x'$ is incidence vector of some $S$-multicut in $G$, i.e. $x' \in \Mcut(G,S)$ and thus, $x \in \Mcut(G,S)$. 
	\end{proof}

\section{Constructing Facets from Facets}\label{sec:constructing_facets}
In \cite{OnTheCutPolytope} an extensive study on the effect of graph operations such as node splittings and edge subdivisions on the max-cut polytope and its facet-defining inequalities has been conducted. Motivated by this, we investigate the same questions for the multicut dominant. Note that in the results from \cite{OnTheCutPolytope} it is always mentioned that certain edges might be added to the newly obtained graph by attaching coefficient $0$ to them. Since for $\Mcut(G,S)$ arbitrary edges can be added by attaching weight $0$ (see Theorems \ref{thm: projection} (iv) and \ref{thm:lifting_inducedSubgraph}), we do not mention this explicitly in each result.

In \Cref{obs: contraction} we have already seen the effect of edge contractions on the multicut dominant.
Now, we investigate the inverse operation:

\begin{thm}[Node splitting]\label{thm: NodeSplitting}
	Let $G=(V,E)$ be a graph, $S \subseteq \binom{V}{2}$, $a^\T x \geq b$ be facet-defining for $\Mcut(G,S)$, and $v \in \supp(a)$ be a node.
	
	Obtain $\widebar{G}=(\widebar{V},\widebar{E})$ as follows: replace $v$ by two adjacent nodes $v_1$ and $v_2$ and distribute the edges incident to $v$ arbitrarily among $v_1$ and $v_2$.
	Furthermore, obtain $\wS$ from $S$ by replacing each pair $\{v,t\} \in S$ independently by a (not necessarily strict) subset of $\{\{v_1,t\},\{v_2,t\}\}$.
	
	Define $\varphi\colon\widebar{E}\setminus\{v_1v_2\} \to E$ by  
	$$\varphi(e)=\begin{cases}
		e, 	&\text{if } v_1,v_2\notin e, \\
		vw,	&\text{if } e=v_iw\  (i=1,2).
	\end{cases} $$
	Let $\omega$ be the value of a minimum $\widebar{S}$-multicut in $\widebar{G} - v_1v_2$ when considering $a_{\varphi(e')}$, $e' \in \E\setminus \{v_1v_2\}$, as edge weights.
	Define $\widebar{a}\in \R^\E$ by
	$$\widebar{a}_e=
	\begin{cases}
		b-\omega,&\text{if $e=v_1v_2$,}\\
		a_{\varphi(e)}, &\text{otherwise.}
	\end{cases} $$
	Then $\widebar{a}^\mathsf{T} x \geq b$ defines a facet of $\Mcut(\G,\widebar{S})$.
	
	In particular, if $a^\T x \geq b$ defines a \nice facet of $\Mcut(G,S)$, so does  $\widebar{a}^\mathsf{T} x \geq b$ for $\Mcut(\G,\widebar{S})$.
\end{thm}
\begin{proof}
	If $a^\mathsf{T}x \geq b$ does not define a \nice facet, utilizing \Cref{thm: projection} (iii) we remove edges $e \in E$ with coefficient $a_e=0$ from $G$ and add them back after splitting by applying \Cref{thm: projection} (iv).
	Since each facet-defining inequality is \nice if its support graph is all $G$, we may
	assume that $a^\mathsf{T}x \geq b$ defines a \nice facet. In this case we prove the \qq{in particular} part.
	
	First, we show $b-\omega \geq 0$.
	Note that there is a one-to-one correspondence between $\wS$-multicuts in $\G$ not containing $v_1v_2$ and $S$-multicut in $G$.
	Thus, since $a^\mathsf{T}x \geq b$ is facet-defining for $\Mcut(G,S)$, there is some $\wS$-multicut $\delta$ in $\G$ with $\sum_{e \in \E\setminus\{v_1v_2\}}a_{\varphi(e)} x^\delta_e=b$.
	Since $\delta$ is also an $\wS$-multicut in $\G-v_1v_2$, this yields $\omega \leq b$. 
	
	Next, we prove validity. Let $\delta$ be an $\wS$-multicut in $\G$.
	If $v_1v_2 \notin \delta$, $\varphi(\delta)$ is a multicut in $G$; thus, $\a^\T x^\delta = a^\T x^{\varphi(\delta)} \geq b$. Otherwise, $\delta\setminus\{v_1v_2\}$ is an $S$-multicut in $\G-v_1v_2$. Hence,
	$$\a^\mathsf{T} x^\delta= (b-\omega)x_{v_1v_2} + \sum_{\substack{e \in \E\\e \neq v_1v_2 }}\a_e x_e \geq (b-\omega) + \omega =b.$$
	It remains to show that $\a^\T x \geq b$ is indeed facet-defining. Since $a^\T x \geq b$ is facet-defining for $\Mcut(G,S)$ there are $S$-multicuts $\delta_1, \dots , \delta_m$ ($m=|E|$) in $G$ such that $x^{\delta_1}, \dots, x^{\delta_m}$ are affinely independent and $a^\T x^{\delta_i}=b$ for each $1 \leq i \leq m$.
	Set $\widebar{\delta}_i= \varphi^{-1}(\delta_i)$.
	Since $\varphi$ is bijective, each $\widebar{\delta}_i$ is well-defined, an $\wS$-multicut in $\G$, and we have $\a^\mathsf{T} x^{\widebar{\delta}_i}=b$.
	Moreover, for an $\wS$-multicut $\delta$ in $\G-v_1v_2$ with $\sum_{e \in \delta}a_e=\omega$ let $\widebar{\delta}=\delta \cup \{v_1v_2\}$; we have $\a^\T x^{\widebar{\delta}}=b$. 
	Since $x^{\widebar{\delta}},x^{\widebar{\delta}_1},\dots,x^{\widebar{\delta}_m}$ are affinely independent $\a^\mathsf{T} x \geq b$ defines a \nice facet. 
\end{proof}

The previous theorem can be utilized to construct new classes of facet-defining inequalities from known facets, as illustrated in the following example.

\begin{ex}\label{ex: splitted_3-claw}
	Consider the graph $G$ shown in \Cref{fig: ex_nodesplit}(a) with terminal pairs $S=\{\{s,t\},\{s,u\},\{t,u\}\}$. 
	We will later see in \Cref{thm: complete-star} that $\sum_{e \in E(G)}x_e \geq 2$ is facet-defining for $\Mcut(G,S)$.
	
	Consider the graph $G_1$ obtained from $G$ by splitting $s$ into $s$ and $s_1$ (cf. \Cref{fig: ex_nodesplit}(b)) and a new set of terminal pairs $S_1=\{\{s_1,t\},\{s,u\},\{t,u\}\}$.
	Then \Cref{thm: NodeSplitting} yields that $\sum_{e \in E(G_1)}x_e \geq 2$ defines a facet of $\Mcut(G_1,S_1)$.
	Now, we obtain $G_2$ from $G_1$ by a splitting of $s$ into $s$ and $s_2$ (cf. \Cref{fig: ex_nodesplit}(c))and setting $S_2=\{\{s_1,t\},\{s_2,u\},\{t,u\}\}$; \Cref{thm: NodeSplitting} yields that $\sum_{e \in E(G_2)}x_e \geq 2$ is facet-defining for $\Mcut(G_2,S_2)$.
	Splitting $t$ and $u$ in the same fashion, we obtain the graph $\widebar G$ shown in \Cref{fig: ex_nodesplit}(d) with 
	$\widebar S=\{\{s_1,t_1\},\{s_2,u_1\},\{t_2,u_2\}\}$.
	By \Cref{thm: NodeSplitting}, $\sum_{e \in E(\widebar G)}x_e \geq 2$ is facet-defining for $\Mcut(\widebar G,\widebar S)$.
\end{ex}
\begin{figure}
	\begin{subfigure}{.5\textwidth}
		\centering
		\begin{tikzpicture}
			\node[circle,draw] (r) at (0,-1) {};
			\node[circle,draw] (s) at (-2,-2) {\footnotesize $s$};
			\node[circle,draw] (t) at ( 0,-2) {\footnotesize $t$};
			\node[circle,draw] (u) at ( 2,-2) {\footnotesize $u$};
			
			\draw (r)--(s);
			\draw (r)--(t);
			\draw (r)--(u);
			
			\draw[blue!50,dashed,thick] (s) edge[out=-40,in=-140] (t);
			\draw[blue!50,dashed,thick] (t)	edge[out=-40,in=-140]	(u);
			\draw[blue!50,dashed,thick] (s)	edge[out=-50,in=-130]	(u);
			
			\useasboundingbox (-3,-3.5)--(3,-3.5)--(3,-.5)--(-3,-.5)--(-3,-3.5);
		\end{tikzpicture}
		\caption{$G$}\label{fig: ex_nodesplita}
	\end{subfigure}
	\begin{subfigure}{.5\textwidth}
		\centering
		\begin{tikzpicture}
			\node[circle,draw] (r) at (0,-1) {};
			\node[circle,draw] (s) at (-2,-2) {\footnotesize $s$};
			\node[circle,draw,inner sep=2.5pt] (s1) at (-2.5,-3) {\footnotesize $s_1$};
			\node[circle,draw] (t) at ( 0,-2) {\footnotesize $t$};
			\node[circle,draw] (u) at ( 2,-2) {\footnotesize $u$};
			
			\draw (r)--(s);
			\draw(s)--(s1);
			\draw (r)--(t);
			\draw (r)--(u);
			
			\draw[blue!50,dashed,thick] (s1) edge[out=-40,in=-90] (t);
			\draw[blue!50,dashed,thick] (t)	edge[out=-40,in=-140]	(u);
			\draw[blue!50,dashed,thick] (s)	edge[out=-50,in=-130]	(u);
			
			\useasboundingbox (-3,-3.5)--(3,-3.5)--(3,-.5)--(-3,-.5)--(-3,-3.5);
		\end{tikzpicture}
		\caption{$G_1$}\label{fig: ex_nodesplitb}
	\end{subfigure}
	
	\begin{subfigure}{.5\textwidth}
		\centering
		\begin{tikzpicture}
			\node[circle,draw] (r) at (0,-1) {};
			\node[circle,draw] (s) at (-2,-2) {\footnotesize $s$};
			\node[circle,draw,inner sep=2.5pt] (s1) at (-2.5,-3) {\footnotesize $s_1$};
			\node[circle,draw,inner sep=2.5pt] (s2) at (-1.5,-3) {\footnotesize $s_2$};
			\node[circle,draw] (t) at ( 0,-2) {\footnotesize $t$};
			\node[circle,draw] (u) at ( 2,-2) {\footnotesize $u$};
			\draw (r)--(s);
			\draw (s)--(s1);
			\draw (s)--(s2);
			\draw (r)--(t);
			\draw (r)--(u);
			
			\draw[blue!50,dashed,thick] (s1) edge[out=-40,in=-90] (t);
			\draw[blue!50,dashed,thick] (t)	edge[out=-40,in=-140]	(u);
			\draw[blue!50,dashed,thick] (s2)	edge[out=-40,in=-90]	(u);
			
			\useasboundingbox (-3,-3.5)--(3,-3.5)--(3,-.5)--(-3,-.5)--(-3,-3.5);
		\end{tikzpicture}
		\caption{$G_2$}\label{fig: ex_nodesplitc}
	\end{subfigure}
	\begin{subfigure}{.5\textwidth}
		\centering
		\begin{tikzpicture}
			\node[circle,draw] (r) at (0,-1) {};
			\node[circle,draw] (s) at (-2,-2) {\footnotesize $s$};
			\node[circle,draw,inner sep=2.5pt] (s1) at (-2.5,-3) {\footnotesize $s_1$};
			\node[circle,draw,inner sep=2.5pt] (s2) at (-1.5,-3) {\footnotesize $s_2$};
			\node[circle,draw] (t) at ( 0,-2) {\footnotesize $t$};
			\node[circle,draw,inner sep=2.5pt] (t1) at (-.5,-3) {\footnotesize $t_1$};
			\node[circle,draw,inner sep=2.5pt] (t2) at ( .5,-3) {\footnotesize $t_2$};
			\node[circle,draw] (u) at ( 2,-2) {\footnotesize $u$};
			\node[circle,draw,inner sep=2.5pt] (u1) at (2.5,-3) {\footnotesize $u_2$};
			\node[circle,draw,inner sep=2.5pt] (u2) at (1.5,-3) {\footnotesize $u_1$};
			\draw (r)--(s);
			\draw (s)--(s1);
			\draw (s)--(s2);
			\draw (r)--(t);
			\draw (r)--(u);
			\draw (t)--(t1);
			\draw (t)--(t2);
			\draw (u)--(u1);
			\draw (u)--(u2);
			
			\draw[blue!50,dashed,thick] (s1) edge[out=-40,in=-140] (t1);
			\draw[blue!50,dashed,thick] (t2)	edge[out=-40,in=-140]	(u1);
			\draw[blue!50,dashed,thick] (s2)	edge[out=-40,in=-140]	(u2);
			
			\useasboundingbox (-3,-3.5)--(3,-3.5)--(3,-.5)--(-3,-.5)--(-3,-3.5);
		\end{tikzpicture}
		\caption{$\widebar G$}\label{fig: ex_nodesplitd}
	\end{subfigure}
	\caption{The graphs from \Cref{ex: splitted_3-claw} in black. Dashed blue connections represent the terminal pairs.}\label{fig: ex_nodesplit}
\end{figure}
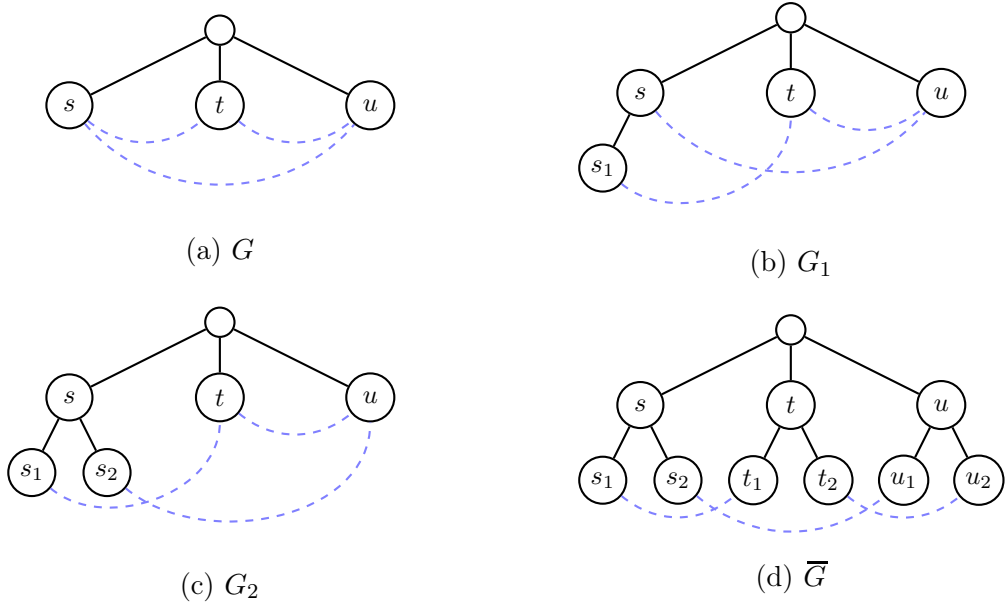

\begin{thm}[Edge subdivision]\label{thm: EdgeSubdivision}
	Let $G=(V,E)$ be a graph, $S \subseteq \binom{V}{2}$, $a^\T x \geq b$ be facet-defining for $\Mcut(G,S)$, and $f \in E$.
	Obtain $\G=(\V,\E)$ from $G$ by subdividing $f$ into $f_1,f_2$. 
	Then,
	\begin{equation}\label{eq: Subdividing_Edge}
		a_f(x_{f_1}+x_{f_2}) + \sum_{e \in E\setminus\{f\}}a_ex_e\geq b
	\end{equation}
	is facet-defining for $\Mcut(\G,S)$.
	
	In particular, if $a^\T x \geq b$ defines a \nice facet of $\Mcut(G,S)$, so does \eqref{eq: Subdividing_Edge} for $\Mcut(\G,S)$.
\end{thm}
\begin{proof}
	Validity of the inequality is straight-forward to verify.
	Let $m=|E|$.
	As in the proof of \Cref{thm: NodeSplitting} we may assume that $a^\T x\geq b$ defines a \nice facet and prove the \qq{in particular} statement for this case.
	
	Since $a^\T x \geq b$ defines a \nice facet of $\Mcut(G,S)$, there exist multicuts $\delta_1, \dots, \delta_m$ in $G$ such that $x^{\delta_1}, \dots, x^{\delta_m}$ are affinely independent and $a^\T x^{\delta_i} =b$ for all $i \in [m]$.
	Now set for each $i \in [m]$ 
	$$\widebar{\delta_i}=
	\begin{cases}
		\delta_i,								&\text{if } f \notin \delta_i,\\
		\left(\delta_i \setminus\{f\} \right) \cup \{f_1\},	&\text{otherwise.}
	\end{cases}
	$$
	By \Cref{lem:nonzero_coeff_support} we may without loss of generality assume $f \in \delta_1$ and set $\delta= \left(\delta_1 \setminus \{f\} \right)\cup \{f_2\}$. Then $\delta, \widebar{\delta}_1, \dots ,\widebar{\delta}_m$ are multicuts in $\G$. Since $x^{\delta_1}, \dots, x^{\delta_m}$ are affinely independent and $x^\delta_{f_2}=1$ whereas $x^{\widebar{\delta}_1}_{f_1}= \dots= x^{\widebar{\delta}_m}_{f_2}=0$, the vectors $x^\delta,x^{\widebar{\delta}_1}, \dots, x^{\widebar{\delta}_m}$ are affinely independent. Moreover, they all satisfy \eqref{eq: Subdividing_Edge} with equality.
	Hence, inequality \eqref{eq: Subdividing_Edge} defines a \nice facet of $\Mcut(\G,S)$.
\end{proof}
Iteratively applying \Cref{thm: EdgeSubdivision} we obtain the following corollary:

\begin{cor}[Replacing an edge by a path]\label{cor: Replace_edge_by_path}
	Let $G=(V,E)$ be a graph, $S \subseteq \binom{V}{2}$, $a^\T x \geq b$ be facet-defining for $\Mcut(G,S)$, and
	$f \in E$.
	Obtain $\widebar{G}=(\widebar{V},\widebar{E})$ by replacing $f$ by a path $P$.
	Then, 
	\begin{equation}\label{eq: ReplaceEdgeByPath}
		a_f\sum_{e \in E(P)}x_e + \sum_{e \in \E\setminus E(P)}a_ex_e\geq b
	\end{equation}
	is facet-defining for $\Mcut(\widebar{G},S)$.
	
	In particular, if $a^\T x \geq b$ defines a \nice facet of $\Mcut(G,S)$, so does \eqref{eq: ReplaceEdgeByPath} for $\Mcut(\G,S)$.
\end{cor}

If we replace an edge $uv \in E(G)$ in a graph $G$ by any connected graph $H$, \Cref{thm: projection} (iv) yields a facet for each facet-defining inequality of $\Mcut(G,S)$ and any $uv$-path in $H$.
Next, we want to study the inverse operation, i.e., the replacement of certain subgraphs by an edge.

\begin{thm}[Replacing a connected graph by an edge]\label{thm: ReplacingSubgraph}
	Let $G=(V,E)$ be a graph and $S \subseteq \binom{V}{2}$. Let $H\subseteq G$ be a connected subgraph of $G$ such that $H$ shares precisely two vertices $s$, $t$ with the rest of $G$, i.e., no edge in $E(G)\setminus E(H)$ is incident to any node in $V(H)\setminus\{s,t\}$.
	Assume there is no terminal in $V(H)\setminus\{s,t\}$.
	Let $a^\mathsf{T}x \geq b$ be facet-defining for $\Mcut(G,S)$ and let $\omega$ be the weight of a minimum $s$-$t$-cut $\delta_{st}$ in $H$ with respect to edge weights given by~$a$.
	Obtain $\G=(\V,\E)$ from $G$ by replacing $H$ by the edge $st$ and define $\a \in \R^{\E}$ by
	$$\a_e=\begin{cases}
		a_e,	&\text{if } e \neq st,\\
		\omega,	&\text{if } e=st.
	\end{cases} $$
	Then, $\a^\mathsf{T}x \geq b$ is facet-defining for $\Mcut(\G,S)$.
\end{thm}

\begin{proof}
	By \Cref{thm: projection} (iii) and (iv), we may remove all edges in $G$ with coefficient $0$, and add such edges in $E(G)\setminus E(H)$ back after the replacement of $H$. Hence, we may assume that $G= \supp(a)$ and $\omega \neq 0$.
	We start by proving validity of the claimed inequality. Let $\delta$ be a multicut in $\G$. If $st \notin \delta$, then $\delta$ is also a multicut in $G$ and thus $\a^\mathsf{T} x^\delta=a^\mathsf{T} x^\delta \geq b$. If $st \in \delta$, the set $(\delta\setminus \{st\}) \cup \delta_{st}$ is an $S$-multicut in $G$ and we have $\a^\mathsf{T}x^\delta=a^\mathsf{T}x^{(\delta\setminus \{st\})} \geq b$.
	
	Now, let $\pi\colon \R^E \to \R^\E$ be the projection given by
	$$\pi(x)_e=\begin{cases}
		x_e,					&\text{if } e \neq st,\\
		\frac{1}{\omega} \sum_{e \in E(H)}x_e,	&\text{if } e=st.
	\end{cases} $$
	Clearly, we have $\a^\T \pi(x^\delta) =b$ for each $S$-multicut $\delta$ in $G$ with $a^\T x^\delta=b$. Thus,
	\begin{align*}
		&\dim(\{\a^\mathsf{T} x =b\}\cap \Mcut(\G,S)) \\
		&\geq \dim(\{a^\mathsf{T} x =b\}\cap\Mcut(G,S)) 	-\left(|E(H)|-1\right) = |\E|.
	\end{align*}
	Moreover, since $a^\T x \geq b$ is facet-defining we have $a \neq \bnull$ and thus, $\a \neq \bnull$.
	Hence, $\a^\mathsf{T} x \geq b$ is facet-defining for $\Mcut(\G,S)$.
\end{proof}
\Cref{cor: Replace_edge_by_path} and \Cref{thm: ReplacingSubgraph} naturally give rise to the following conjecture:
\begin{conj}\label{conj: 2component}
	Let $G=(V,E)$ be a graph, $S \subseteq \binom{V}{2}$, and $\G$ be obtained from $G$ by replacing an edge by a connected graph. Then any facet-defining inequality of $\Mcut(\G,S)$ is either an edge inequality, or
	obtained from facet-defining inequalities of $\Mcut(G,S)$ by applying \Cref{cor: Replace_edge_by_path}.
\end{conj}

%	\begin{thm}[Replacing a path by an edge]
%		Let $G=(V,E)$, $S$ be a clean set of terminals and $a^\T x \geq b$ be facet-defining for $\Mcut(G,S)$ with $\{a^\T x =b\} \neq \{x_e=0\}$ for all $e \in E$. Furthermore, let $P \subseteq G$ be an induced path such that no internal node of $P$ is a terminal and $a_e=M$ for all $e \in P$. Obtain $\G$ from $G$ by replacing $P$ by a newly introduced edge $p$. Set $a\in \R^\E$ by 
%		$$\a_e= \begin{cases}
%			a_e,		&\text{for } e\neq p\\
%			M			&\text{for } e=p.
%		\end{cases} $$
%		Then, $\a^\T x \geq b$ is facet-defining for $\Mcut(\G,S)$.
%	\end{thm}
%	\begin{proof}
%		Consider the projection $\pi \colon \R^E \to \R^\E$, 
%		$$\pi(x)_e=\begin{cases}
%			x_e					&\text{for } e \in E \cap \E,\\
%			\sum_{e \in P}x_e	&\text{for } e = p.
%		\end{cases}$$
%		Clearly we have $\pi(x) \in \Mcut(\G,S)$ for each $x \in \Mcut(G,S)$. 
%		Thus, the inequality is valid for $\Mcut(\G,S)$.
%		
%		By \Cref{cor:coeff_on_path} we have $\a^\T \pi(x^\delta) =b$ for each $S$-multicut $\delta$ in $G$ with $a^\T x^\delta=b$. Thus,
%		\begin{align*}
%			&\dim(\{\a^\mathsf{T} x =b\}\cap \Mcut(\G,S)) \\
%			\geq &\dim(\{a^\mathsf{T} x =b\}\cap\Mcut(G,S)) 	-\left(|P|-1\right) = |\E|.
%		\end{align*}
%		Moreover, since $a^\T x \geq b$ is facet-defining and not an edge inequality we have $a \neq \bnull$  and thus, $\a \neq \bnull$.
%		Hence, $\a^\mathsf{T} x \geq b$ is facet-defining for $\Mcut(\G,S)$.
%	\end{proof}

\section{Star Inequalities}\label{sec:star-ineq}
	
	In this section we investigate the multicut dominant of $K_{1,n}$ with respect to specific sets of terminal pairs. For $K_{1,3}$ with leaves $W$ and $S=\binom{W}{2}$, it is known that edge- and path inequalities do not suffice to give a complete description of $\Mcut(K_{1,3},S)$ \cite[Section 1]{NP-hard_trees}.
	Since $S$ can be viewed as a triangle on the leaves, there are two natural generalizations of this instance to $K_{1,n}$: The terminal pairs may form a cycle on the leaves or the terminal pairs may form a complete graph on the leaves.
	We discuss facets based on both generalizations.
	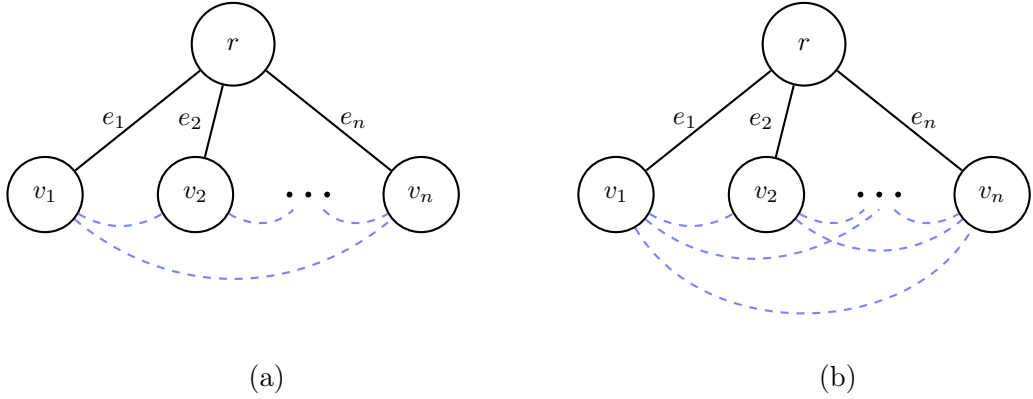
\begin{figure}
		\begin{subfigure}{.5\textwidth}
			\begin{tikzpicture}
				\node[circle,draw,minimum size=1.1cm] (r) at (0,0) {\footnotesize $r$};
				
				\node[circle,draw,minimum size=1cm] (v1) at (-2.5,-2) {\footnotesize $v_1$};
				\node[circle,draw,minimum size=1cm] (v2) at (-.5,-2) {\footnotesize $v_2$};
				\node[circle,draw,minimum size=1cm] (v3) at (2.5,-2) {\footnotesize $v_n$};
				\draw (r)--node[above,left]{\footnotesize{$e_1$}}(v1);
				\draw (r)--node[above,left]{\footnotesize{$e_2$}}(v2);
				\draw (r)--node[above,right]{\footnotesize{$e_n$}}(v3);
				
				\draw[blue!50,dashed,thick] (v1) edge[out=-30,in=-150] (v2);
				\draw[blue!50,dashed,thick] (v2)	edge[out=-30,in=-135]	(.8,-2.2);
				\draw[blue!50,dashed,thick] (1.2,-2.2)	edge[out=-45,in=-150]	(v3);
				\draw[blue!50,dashed,thick] (v1)	edge[out=-40,in=-140]	(v3);
				
				\node[circle,fill,scale=.2] (d1) at (.75,-2) {};
				\node[circle,fill,scale=.2] (d2) at (1,-2) {};
				\node[circle,fill,scale=.2] (d3) at (1.25,-2) {};
				
				\useasboundingbox (-3.3,-4)--(-3.3,1)--(3.3,1)--(3.3,-4)--(-3.3,-4);
			\end{tikzpicture}\caption{}
		\end{subfigure}
		\begin{subfigure}{.5\textwidth}
			\begin{tikzpicture}
				\node[circle,draw,minimum size=1.1cm] (r) at (0,0) {\footnotesize $r$};
				
				\node[circle,draw,minimum size=1cm] (v1) at (-2.5,-2) {\footnotesize $v_1$};
				\node[circle,draw,minimum size=1cm] (v2) at (-.5,-2) {\footnotesize $v_2$};
				\node[circle,draw,minimum size=1cm] (v3) at (2.5,-2) {\footnotesize $v_n$};
				\draw (r)--node[above,left]{\footnotesize{$e_1$}}(v1);
				\draw (r)--node[above,left]{\footnotesize{$e_2$}}(v2);
				\draw (r)--node[above,right]{\footnotesize{$e_n$}}(v3);
				
				\node[circle,fill,scale=.2] (d1) at (.75,-2) {};
				\node[circle,fill,scale=.2] (d2) at (1,-2) {};
				\node[circle,fill,scale=.2] (d3) at (1.25,-2) {};
				
				\draw[blue!50,dashed,thick] (v1) edge[out=-30,in=-150] (v2);
				\draw[blue!50,dashed,thick] (v2)	edge[out=-30,in=-135]	(.8,-2.2);
				\draw[blue!50,dashed,thick] (1.2,-2.2)	edge[out=-45,in=-150]	(v3);
				\draw[blue!50,dashed,thick] (v2)	edge[out=-40,in=-140]	(v3);
				\draw[blue!50,dashed,thick] (v1)	edge[out=-40,in=-140]	(1,-2.2);
				\draw[blue!50,dashed,thick] (v1)	edge[out=-60,in=-120]	(v3);
				
				\useasboundingbox (-3.3,-4)--(-3.3,1)--(3.3,1)--(3.3,-4)--(-3.3,-4);
			\end{tikzpicture}\caption{}
		\end{subfigure}
		\caption{The graphs from Theorems \ref{thm: circular_star_facet} and \ref{thm: complete-star} in black. Dashed blue connections represent the terminal pairs.}\label{fig: star_inequalities}
	\end{figure}

	Throughout this section we consider a star $K_{1,n}{=}(\{r,v_0,\dots,v_{n-1}\},\{0 ,\dots, n{-}1\})$ with $i=rv_i$ for $0 \leq i < n$.
	\begin{thm}\label{thm: circular_star_facet}
		Let $n \geq 3$ be odd, consider $K_{1,n}$, and let $S=\{\{v_i,v_{(i+1) \bmod n}\}:0 \leq i <n\}$ (cf. \Cref{fig: star_inequalities}(a)). 
		Then, the \emph{circular $n$-star inequality}
		$\sum_{e \in E(K_{1,n})}x_e \geq \lceil \frac{n}{2} \rceil$ defines a \nice facet of $\Mcut(K_{1,n},S)$.
	\end{thm}

	\begin{proof}
		It is straight-forward to verify validity of the inequality. We now prove that it is indeed facet-defining.
		For $0 \leq k <n$ set ${\delta^n_k =\{(k+2i)\bmod n: 0 \leq i \leq \lfloor \frac{n}{2}\rfloor\}}$.
		Then, the $\delta^n_k$ are pairwise distinct multicuts in $K_{1,n}$ as each such set contains precisely two consecutive (but different) edges of $K_{1,n}$.
		Clearly, it holds $\sum_{j=1}^{n-1}x^{\delta^n_k}_j = \lceil \frac{n}{2} \rceil$. 
		We show by induction that $x^{\delta^n_0}, \dots, x^{\delta^n_{n-1}}$ are affinely independent.
		Clearly, this is true for $n=3$.
		
		For $2 \leq i <n$ we have $\pi (x^{\delta^n_k})= x^{\delta^{n-2}_{k-2}}$ where $\pi \colon\R^n \to \R^{n-2}$ is the orthogonal projection $(x_0, \dots, x_{n-1}) \mapsto (x_2, \dots, x_{n-1})$. 
		Thus, by induction we know that $x^{\delta^n_2}, \dots, x^{\delta^n_{n-1}}$ are affinely independent.
		Let $\lambda_0, \dots, \lambda_{n-1} \in \R_{\geq 0}$ with  $\sum_{k=0}^{n-1} \lambda_k=0$ and $0= \sum_{k=0}^{n-1} \lambda_k x^{\delta^n_k}$.
		Comparing the coefficients of the first three entries of this affine combination, we obtain
		\begin{alignat*}{2}
			\lambda_0 + &\sum_{k=1}^{\mathclap{\lfloor n/2\rfloor}} \lambda_{2k-1}	&&=0 \\
			\lambda_1 + &\sum_{k=1}^{\mathclap{\lfloor n/2\rfloor}} \lambda_{2k\phantom{+1}}	&&=0 \\
			\lambda_0 + \lambda_2 + &\sum_{k=2}^{\mathclap{\lfloor n/2\rfloor}} \lambda_{2k-1}	&&=0 
		\end{alignat*}
		Subtracting the first from the third inequality we obtain $\lambda_2- \lambda_1 =0$ and summing the first and second inequality we obtain 
		$0=\lambda_1 + \sum_{i=0}^{n-1} \lambda_i=\lambda _1$. Thus, $\lambda_1= \lambda_2=0$. Since $x^{\delta^n_3}, \dots, x^{\delta^n_n}$ are affinely independent, we conclude $\lambda_1=\dots=\lambda_n=0$.
		
		The defined facet is \nice since the inequality is minimally supported on $K_{1,n}$.
	\end{proof}

	Note that considering even $n$ in the scenario of the previous theorem, the corresponding inequality $\sum_{e \in E(K_{1,n})}x_e \geq \frac{n}{2}$ would be dominated by path-inequalities and thus not be facet-defining.
	However, the inequalities from the previous theorem together with edge- and path inequalities suffice to completely describe the multicut dominant for arbitrary $n$.

%	\begin{prop}\cite{HellerTompkins,FulkersonGross}\label{prop:totally-unimodular_conditions}\todo{Martina: \qq{genaue Referenzen} Kriterien in Beweis aufführen}
%		Let $A$ be an $(m \times n)$-matrix such that every entry of $A$ equals $-1$, $0$, or $1$. Then $A$ is totally unimodular if one of the following conditions holds:
%		\begin{enumerate}[(i)]
%			\item The columns of $A$ can be partitioned into two disjoint sets $B$ and $C$ such that
%				\begin{itemize}
%					\item Every row of $A$ contains at most two non-zero entries,
%					\item if two non-zero entries in the same row have the same sign, the column of one is in $B$ and the other in $C$,
%					\item If two non-zero entries in the same row have opposite sign, either the columns of both are in $B$ or the columns of both are in $C$.
%				\end{itemize}
%			\item $A$ can be permuted into a $0$-$1$-matrix in which for every row the $1$s appear consecutively.
%		\end{enumerate}
%	\end{prop}
	
	\begin{cor}\label{thm: complete_description_circular-claw}
		Let $n \geq 3$, consider $K_{1,n}$, and let $S=\{\{v_i,v_{(i+1) \bmod n}\}:0 \leq i <n\}$. 
		Then the  $\Mcut(K_{1,n},S)$ is completely described by the inequalities
		\begin{alignat*}{2}
							x_e & \geq 0  \hspace{1cm} 	&&\text{for all } e \in E,\\
			x_i+x_{(i+1)\bmod n} &\geq 1 &&\text{for } 0 \leq i <n,\\
			\sum_{e \in E}x_e	&\geq \left\lceil \frac{n}{2}\right\rceil ,&& 
			\end{alignat*}
		where the last inequality can be omitted if and only if $n$ is odd.
	\end{cor}
	\begin{proof}
		A matrix $A$ is \emph{totally unimodular} if each square submatrix has determinant $-1$, $0$, or $1$. 
		It is well known that if $A$ is totally unimodular, and $b \in \Z^m$, all vertices of the polyhedron $\{x \in \R^d: Ax \geq b\}$ are integral.
		
		Let $\mathcal P=\{Ax \geq b\}$ be the polyhedron defined by the claimed inequalities.
		In the following we prove that each vertex of $\mathcal P$ is integral, which yields $\mathcal P=\Mcut(G,S)$.
		
		First assume $n$ is even. We show that $A$ is totally unimodular by using a characterization due to \cite[Appendix]{HellerTompkins}:
		Let $A$ be an $(m \times n)$-matrix such that every entry of $A$ equals $-1$, $0$, or $1$. Then $A$ is totally unimodular if the columns of $A$ can be partitioned into two disjoint sets $B$ and $C$ such that
			\begin{itemize}
				\item every row of $A$ contains at most two non-zero entries;
				\item if two non-zero entries in the same row have the same sign, the column of one is in $B$ and the other in $C$; and
				\item if two non-zero entries in the same row have opposite signs, either both columns are in $B$ or both columns are in $C$.
			\end{itemize}
		By partitioning the columns of $A$ based on the parity of their index, the previous characterization yields that $A$ is totally unimodular. Thus, all vertices of $\mathcal P$ are integral.
		
		Now, assume $n$ is odd. Let $v \in \mathcal P$ be a vertex.
		We call the above path inequalities \emph{$2$-path inequalities} due to the length of their support.
		Note that the only point in $\mathbb R^E$ satisfying all $2$-path inequalities with equality is $(\frac 1 2, \dots, \frac 1 2)$, which is not contained in $\mathcal P$.
		Thus, there is a $2$-path inequality that is not satisfied by $v$ with equality.
		By symmetry we may assume that this $2$-path inequality is $x_n+x_1 \geq 1$.
		Let $A'$ and $b'$ be obtained from $A$ and $b$, respectively, by deleting the row corresponding to this inequality.
		Then, $v$ is a vertex of the polyhedron $\{A'x \geq b'\}$.
		We prove $v \in \Z^{E(K_{1,n})}$ by showing that $A'$ is totally unimodular.
		To this end, we use a characterization from \cite[Section 8]{FulkersonGross}:
		Let $A$ be an $(m \times n)$-matrix such that every entry of $A$ equals $0$ or $1$. 
		Then $A$ is totally unimodular if its colums can be permuted such that for every row the $1$s appear consecutively.
		
		$A'$ is (up to permutation of rows) of the form
		$$\begin{pmatrix}´
			1		&	1	&	0	&	 \cdots	&	0	\\
			0		&	1	&	1	& 	\ddots	& \vdots\\
			\vdots	&\ddots	& \ddots&	\ddots	& 	0	\\
			0 		&\cdots	&	0	&	1 		&	1	\\
			1		&	0	&\cdots	&	\cdots	&	0	\\
			0		&	1	&\ddots	&			&\vdots	\\
			\vdots	&\ddots	&\ddots	&´	\ddots	&\vdots	\\
			\vdots	&		&\ddots	&´\ddots	&	0	\\
			0		&\cdots	&\cdots	&	0		&	1	\\
			1		&	1	&\cdots	&	1		&	1
		\end{pmatrix} $$
		and thus, totally unimodular.
	\end{proof}
	
	We now turn our attention to the second mentioned class of instances. cf. \Cref{fig: star_inequalities}(b).
	
	\begin{thm}\label{thm: complete-star}
		Let $n \in \N$, consider $K_{1,n}$ and let $S=\{\{v_i,v_j\}:0 \leq i < j <n\}$. 
		Then, the \emph{complete $n$-star inequality}
		$\sum_{e \in E(K_{1,n})}x_e \geq n-1$ defines a \nice facet of $\Mcut(K_{1,n},S)$.
	\end{thm}

	\begin{proof}
		It is straight-forward to verify that the inequality is valid for $\Mcut(G,S)$. Moreover, the vertices of $\Mcut(K_{1,n},S)$ are $x^{E \setminus \{i\}}$ for all $0\leq i <n$. These are affinely independent. Since all of these vectors satisfy $\sum_{i=0}^{n-1}x_i = n-1$, the inequality is facet-defining.
		
		The defined facet is \nice since the inequality is minimally supported on $K_{1,n}$. 
	\end{proof}
	Together with \Cref{thm: projection} (iv) this theorem gives a facet-defining inequality for $\Mcut(K_{1,n},S)$ with $S=\{\{s,t\}: s,t \text{ are leafs}\}$ corresponding to each $K_{1,k} \subseteq K_{1,n}$ with the induced sets of terminals.
	Thus, the previous theorem gives a large number of facet-defining inequalities for $\Mcut(K_{1,n},S)$.
	Motivated by this we generalize these inequalities further in \Cref{sec:tree-ineq} by considering more general trees instead of stars.
	
	By \cite{NP-hard_trees} the minimum multicut problem is NP-hard on trees (in fact already on stars).
	Motivated by this, we present a polynomial time separation algorithm for generalizations of circular- and complete n-star inequalities for the multicut dominants when the input graph is restricted to a tree.
	We call the inequalities obtained from these facet-defining inequalities by applying \Cref{cor: Replace_edge_by_path} \emph{subdivided circular $n$-star inequalities} and \emph{subdivided complete $n$-star inequalities}, respectively.
	By \Cref{thm: projection} (iv) they yield facet-defining inequalities for each graph containing the according subdivision of $K_{1,n}$ with respective sets of terminal pairs.
	
	\begin{cor}\label{cor: Separation_Stars}
		Let $k \in \N$ be fixed.
		Given an input graph $G$ that is a tree and a set $S \subseteq \binom{V(G)}{2}$ of terminals, we can enumerate all facet-defining subdivided circular $k$-stars and subdivided complete $k$-stars inequalities for $\Mcut(G,S)$ in polynomial time.\\
		In particular, these inequalities can thus be separated in polynomial time.
	\end{cor}
	
	\begin{proof}
		We prove that all facet-defining subdivided complete $k$-star inequalities can be enumerated in polynomial time.
		This can then be shown analogously for subdivided circular $k$-star inequalities.
		Checking each enumerated inequality individually yields a simple separation routine.
		
		Let $G=(V,E)$ be a tree with $|V|=n$.
		There are $n \cdot \binom{n-1}{k} \in \mathcal{O}(n^{k+1})$ choices for a root $r \in V$ and nodes $v_1, \dots , v_k \in V \setminus \{r\}$. 
		We check in linear time (in $|V|+|S|$) whether these nodes form the root and leaves of a $K_{1,k}$ subdivision in $G$ by searching for the unique $r$-$v_i$-paths in $G$ while checking whether these paths are disjoint and no terminal pair containing a node different from $v_1, \dots, v_k$ is induced.
		Then, we can verify whether the leaves induce the necessary terminal pairs in $S$.
		Hence, we obtain an overall runtime of $\mathcal{O}((|V|+|S|)^{k+2})$
	\end{proof}

	Although enumerating all such inequalities might not be very practical, this result should be considered as a proof of concept.
	We are convinced that there are more efficient separation routines for these inequalities using more sophisticated algorithmic approaches.
	However, this discussion would be out of scope for this work.

\section{Tree Inequalities}\label{sec:tree-ineq}
	As we saw in \Cref{ex: splitted_3-claw}, the star inequalities can be generalized to facet-defining inequalities on trees by applying node splits. In the following, we further investigate these inequalities.
	
	Throughout this section we consider the graph $T_n$, as showcased in \Cref{fig: Gn}: $T_n$ is a rooted tree on $n^2+1$ nodes: The root $r$ has $n$ children $v_1, \dots, v_n$ and there are leaves $s_{i,j}$, $t_{i,j}$ ($1 \leq i < j \leq n$) such that $s_{i,j}$ is a child of $v_i$ and $t_{i,j}$ is a child of $v_j$. For $i,j \in [n]$, we set $e_i=r v_i$, $e_{i,j}=v_is_{ij}$, and $f_{i,j}=v_jt_{ij}$. Moreover, we let $L_1^n=\{e_1, \dots, e_n\}$ and $L_2^n=\{e_{i,j},f_{i,j}: 1 \leq i < j \leq n\}$.
	Finally, let $S_n=\{\{s_{i,j},t_{i,j}\}:1 \leq i < j \leq n\}$.
	Observe that $|L_1^n|=n$ and $|L_2^n|=2 \binom{n}{2}$; $G$ thus has precisely $n^2$ edges.
	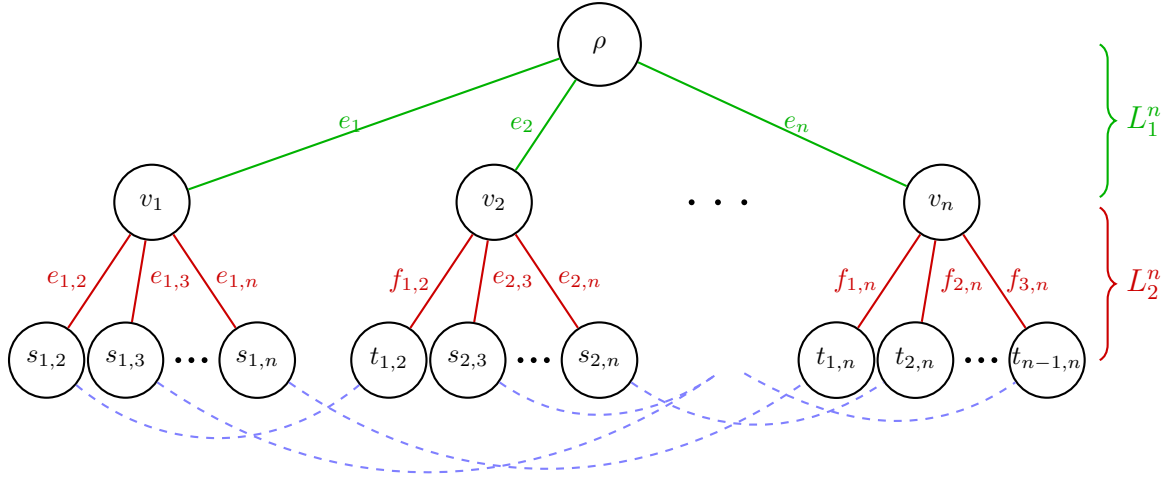
\begin{figure}
		\begin{tikzpicture}[scale=.7]
			\node[circle,draw,minimum size=1.1cm] (r) at (1,0) {\footnotesize $\rho$};
			
			\node[circle,draw,minimum size=1cm] (v1) at (-7.5,-3) {\footnotesize $v_1$};
			\node[circle,draw,minimum size=1cm] (v2) at (-1,-3) {\footnotesize $v_2$};
			\node[circle,draw,minimum size=1cm] (vn) at (7.5,-3) {\footnotesize $v_n$};
			\draw[green!70!black] (r)--node[above,left]{\footnotesize{$e_1$}}(v1);
			\draw[green!70!black] (r)--node[above,left]{\footnotesize{$e_2$}}(v2);
			\draw[green!70!black] (r)--node[above,right]{\footnotesize{$e_n$}}(vn);
			
			\node[circle,draw,minimum size=1cm] (s12) at (-9.5,-6) {\footnotesize $s_{1,2}$};
			\node[circle,draw,minimum size=1cm] (s13) at (-8,-6) {\footnotesize $s_{1,3}$};
			\node[circle,draw,minimum size=1cm] (s1n) at (-5.5,-6) {\footnotesize $s_{1,n}$};
			\draw[red!80!black] (v1) -- node[above,left]{\footnotesize{$e_{1,2}$}} (s12);
			\draw[red!80!black] (v1) -- node[above,right]{\footnotesize{$e_{1,3}$}} (s13);
			\draw[red!80!black] (v1) -- node[above,right]{\footnotesize{$e_{1,n}$}} (s1n);
			
			\node[circle,draw,minimum size=1cm] (t12) at (-3,-6) {\footnotesize $t_{1,2}$};
			\node[circle,draw,minimum size=1cm] (s23) at (-1.5,-6) {\footnotesize $s_{2,3}$};
			\node[circle,draw,minimum size=1cm] (s2n) at (1,-6) {\footnotesize $s_{2,n}$};
			\draw[red!80!black] (v2) -- node[above,left]{\footnotesize{$f_{1,2}$}} (t12);
			\draw[red!80!black] (v2) -- node[above,right]{\footnotesize{$e_{2,3}$}} (s23);
			\draw[red!80!black] (v2) -- node[above,right]{\footnotesize{$e_{2,n}$}} (s2n);
			
			\node[circle,draw,minimum size=1cm] (t1n) at (5.5,-6) {\footnotesize $t_{1,n}$};
			\node[circle,draw,minimum size=1cm] (t2n) at (7,-6) {\footnotesize $t_{2,n}$};
			\node[circle,draw,inner sep=-2mm,minimum size=1cm] (tn-1n) at (9.5,-6) {\footnotesize $t_{n{-}1,n}$};
			\draw[red!80!black] (vn) -- node[above,left]{\footnotesize{$f_{1,n}$}} (t1n);
			\draw[red!80!black] (vn) -- node[above,right]{\footnotesize{$f_{2,n}$}} (t2n);
			\draw[red!80!black] (vn) -- node[above,right]{\footnotesize{$f_{3,n}$}} (tn-1n);
			
			\node[circle,fill,scale=.2] (d1) at (3.25,-3) {};
			\node[circle,fill,scale=.2] (d2) at (3.75,-3) {};
			\node[circle,fill,scale=.2] (d3) at (2.75,-3) {};
			
			\node[circle,fill,scale=.2] (d4) at (-6.75,-6) {};
			\node[circle,fill,scale=.2] (d5) at (-6.5,-6) {};
			\node[circle,fill,scale=.2] (d6) at (-7,-6) {};
			
			\node[circle,fill,scale=.2] (d9) at (0,-6) {};
			\node[circle,fill,scale=.2] (d7) at (-0.25,-6) {};
			\node[circle,fill,scale=.2] (d8) at (-0.5,-6) {};
			
			\node[circle,fill,scale=.2] (d10) at (8,-6) {};
			\node[circle,fill,scale=.2] (d11) at (8.25,-6) {};
			\node[circle,fill,scale=.2] (d12) at (8.5,-6) {};
			
			\draw[green!70!black,thick,decorate,decoration={brace,amplitude=6pt}] (10.5,0) -- (10.5,-2.9) node[midway, right,xshift=2mm]{$\color{green!70!black}L_1^n$};
			\draw[thick,red!80!black,decorate,decoration={brace,amplitude=6pt}] (10.5,-3.1) -- (10.5,-6) node[midway, right,xshift=2mm]{$\color{red!80!black}L_2^n$};
		
			\draw[blue!50,dashed]         (s12)edge[out=-40,in=-140](t12);
			\draw[blue!50,dashed]         (s1n)edge[out=-35,in=-145](t1n);
			\draw[blue!50,dashed]         (s2n)edge[out=-35,in=-145](t2n);
			\draw[blue!50,dashed]         (s13)edge[out=-35,in=-145](3.25,-6.25);
			\draw[blue!50,dashed]         (s23)edge[out=-35,in=-145](3.25,-6.25);
			\draw[blue!50,dashed] (3.75,-6.25) edge[out=-30,in=-145](tn-1n);
		\end{tikzpicture}
		\caption{The tree $T_n$. Edges in $L_1^n$ are green, edges in $L_2^n$ are red. Dashed blue connections visualize $S_n$}\label{fig: Gn}
	\end{figure}
	The main goal of this section is to prove the following theorem:
	\begin{thm}\label{thm: complete_tree_ineq}
		For all $n >k \geq 2$, the \emph{$(n,k)$-tree inequalities}
		\begin{equation}\label{eq: complete_tree}
			(n-k) \sum_{e \in L_1^n}x_e + \sum_{e \in L_2^n}x_e \geq k (n-k)+ \binom{n-k}{2}
		\end{equation}
	
		define \nice facets of $\Mcut(T_n,S_n)$.
	\end{thm}
	To prove this theorem we first prove two auxiliary lemmata.
	\begin{lem}\label{lem: complete_tree_ineq-valid}
		For $n >k \geq 2$, inequality \eqref{eq: complete_tree} is valid for $\Mcut(T_n,S_n)$.
		In particular, the solutions for which \eqref{eq: complete_tree} is tight are precisely the minimal multicuts~$\delta$ in $G$ with $|\delta \cap L_1^n|\in\{k-1,k\}$.
	\end{lem}
	\begin{proof}
		Let $\delta$ be a minimal multicut with $|\delta \cap L_1^n|=\ell$. Since there are $\binom{n-\ell}{2}$ terminal pairs not separated by $\delta \cap L_1^n$ and the removal of an edge in $L_2^n$ separates at most one of those pairs we have $|\delta \cap L_2^n| =\binom{n-\ell}{2}$. Thus,
		\begin{align*}
			&(n-k) \sum_{e \in L_1^n}x^\delta_e + \sum_{e \in L_2^n}x^\delta_e -\left( k (n-k)+ \binom{n-k}{2}\right)\\
			=\ &\ell(n-k)+\binom{n-\ell}{2}- k (n-k)- \binom{n-k}{2}	\\
			=\ &\frac{1}{2}(k-\ell-1)(k-\ell)\geq 0.
		\end{align*}
		Where the last inequality holds since $k, \ell \in \N$.
		The \emph{in particular} part follows since the above inequality is satisfied with equality if and only if $\ell \in \{k-1,k\}$.
	\end{proof}
	
	The following lemma considers the case $k=2$ of \eqref{eq: complete_tree}.
	
	\begin{lem}\label{lem: complete_tree_ineq-base_case}
		For $n \geq 3$ and $k=2$, inequality \eqref{eq: complete_tree} is facet-defining for $\Mcut(T_n,S_n)$.
	\end{lem}
	\begin{proof}
		For the reader's ease, we rewrite the inequality under consideration as
		\begin{equation}\label{eq: complete_tree-k=2}
		(n-2) \sum_{e \in L_1^n}x_e + \sum_{e \in L_2^n}x_e \geq 2 (n-2)+ \binom{n-2}{2}
		\end{equation}
		Validity of \eqref{eq: complete_tree-k=2} is shown in \Cref{lem: complete_tree_ineq-valid}. It remains to verify that the inequality is indeed facet-defining.
		To this end, we show by induction over $n$ that there are $n^2$ $S_n$-multicuts with affinely independent incidence-vectors satisfying \eqref{eq: complete_tree-k=2} with equality.
		For $n=3$, this follows from \Cref{ex: splitted_3-claw}.
		
		Now, let $n \geq 4$. By induction there exist $S_{n-1}$-multicuts $\delta'_1, \dots, \delta'_{(n-1)^2}$ in $T_{n-1}$ such that $\{x^{\delta'_i}: 1 \leq i \leq (n-1)^2 \}$ is an affine independent set and all $x^{\delta'_i}$ satisfy the equality
		 $((n-1)-2) \sum_{e \in L_1^{n-1}}x^{\delta'_i}_e + \sum_{e \in L_2^{n-1}}x^{\delta'_i}_e = 2 ((n-1)-2)+ \binom{(n-1)-2}{2}$.

		Let $\delta_i=\delta'_i\cup \{e_{j,n}: j \in [n],\ e_j \notin \delta'_i\}$ for $1 \leq i \leq (n-1)^2$.
		Since $x^{\delta'_1}, \dots, x^{\delta'_{(n-1)^2}}$ are affinely independent and $|E(T_{n-1})|=(n-1)^2$, for each $\ell \in [n-1]$ there is some $i_\ell \in [(n-1)^2]$ with $e_\ell \notin \delta'_{i_\ell}$.
		Setting $\widehat{\delta}_\ell=\delta'_{i_\ell} \cup\{e_{j,n}: j \in [n]\setminus\{\ell\},\ e_j \notin \delta'_{i_\ell}\}\cup\{f_{\ell,n}\}$ for $1 \leq \ell <n$ the set
		$$A=\Big\{x^{\delta_i}:1 \leq i \leq (n-1)^2 \Big\}\cup \Big\{x^{\widehat{\delta}_\ell}:1 \leq \ell <n \Big\}$$
		is affinely independent and each $x \in A$, attains equality in \eqref{eq: complete_tree-k=2}.
		Now, for $1 \leq i <n$ let $\gamma_i= \{e_i,e_n\} \cup \{e_{a,b}: a,b \neq i,\ 1 \leq a <b <n\}$ and $\gamma_n=\{e_n\} \cup \{e_{a,b}: 1 \leq a < b <n\}$. 
		Then $x^{\gamma_i}$ attains equality in \eqref{eq: complete_tree-k=2}.
		We prove that $A \cup \{x^{\gamma_1},\dots ,x^{\gamma_n}\}$ is affinely independent. Since $|A \cup \{x^{\gamma_1},\dots ,x^{\gamma_n}\}|=n^2$ this yields the claim.
		Note that $A \subseteq \{x_n=0\}$ and $x^{\gamma_1} \notin \{x_n=0\}$. Thus, $A \cup \{x^{\gamma_1}\}$ is affinely independent.
		Now, assume that $A\cup \{x^{\gamma_1},\dots ,x^{\gamma_k}\}$ is affinely independent. Let
		$\mathcal{H}=\left\{\sum_{i=1}^k x_i + (k-1)x_n + \sum_{i=1}^k (x_{e_{i,n}}+x_{f_{i,n}})=k\right\}$. By construction, we have
		$$A\cup \{x^{\gamma_1},\dots ,x^{\gamma_k}\} \subseteq \mathcal H
			\qquad \text{and} \qquad x^{\gamma_{k+1}}\notin  \mathcal H.$$
		Thus, $A\cup \{x^{\gamma_1},\dots ,x^{\gamma_{k+1}}\}$ is affinely independent and the claim follows by induction.
	\end{proof}
	Given the previous two lemmata, we can now prove the main theorem of this section.
	As a tool we use the following simple observation from linear algebra:
	\begin{obs}\label{lem: linear_independent_subsets}
		Let $n > k \geq 1$ and let $\mathds{1}_i \in \R^n$ be the $i$-th unit-vector. Then, there exist $M_1, \dots ,M_n \in \binom{[n]}{k}$ such that $\sum_{i \in M_1}\mathds{1}_i, \dots,\sum_{i \in M_n}\mathds{1}_i$ are linearly independent.
	\end{obs}
%	\begin{proof}
%		We prove the statement by induction on $n$. For $n=2$ the claim is trivial. So, let $n \geq 3$. 
%		If $k=n-1$ note that the Matrix $A=(a_{i,j})_{i,j}$ with 
%		$$a_{i,j}=\begin{cases}
%			0,&\text{if $i=j$,}\\
%			1,&\text{else}
%		\end{cases}$$ has 
%		\begin{align*}
%			\det A&= 
%			\det \begin{pmatrix}
%					n-1	&\cdots &\cdots &\cdots &	n-1	\\
%					1	&	0	&	1	&\cdots &	1	\\
%				\vdots	&\ddots	&\ddots&\ddots	&\vdots	\\
%				\vdots	&		&\ddots&\ddots	&	1	\\
%					1	&\cdots	&\cdots&	1	&	0	
%			\end{pmatrix}
%			=(n-1)\det \begin{pmatrix}
%					1	&\cdots &\cdots &\cdots &	1	\\
%					1	&	0	&	1	&\cdots &	1	\\
%				\vdots	&\ddots	&\ddots&\ddots	&\vdots	\\
%				\vdots	&		&\ddots&\ddots	&	1	\\
%					1	&\cdots	&\cdots&	1	&	0	
%			\end{pmatrix}\\
%			&=(n-1)\det \begin{pmatrix}
%				1	&\cdots &\cdots &\cdots &	1	\\
%				0	&	-1	&	0	&\cdots &	0	\\
%			\vdots	&\ddots	&\ddots&\ddots	&\vdots	\\
%			\vdots	&		&\ddots&\ddots	&	0	\\
%				0	&\cdots	&\cdots&	0	&	-1	
%			\end{pmatrix}=(n-1)(-1)^{n-1}
%		\end{align*}
%		yielding the claim. Hence, let $1 \leq k \leq n-2$. 
%		By induction there exist subsets $S_1, \dots, S_{n-1} \in \binom{[n-1]}{k}$ such that  $\sum_{i \in S_1}b_i, \dots,\sum_{i \in S_{n-1}}b_i$ are linearly independent.
%		Since $n \notin S_i$ for $1 \leq i \leq n-1$, setting $S_n=\{n-k+1,n-k+2, \dots, n\}$ it is easy to see that $\sum_{i \in S_1}b_i, \dots,\sum_{i \in S_{n}}b_i $ are linearly independent.
%	\end{proof}
	
	\begin{proof}[Proof of Theorem \ref{thm: complete_tree_ineq}]
		Validity of \eqref{eq: complete_tree} is proven in \Cref{lem: complete_tree_ineq-valid}.
		From \Cref{lem: complete_tree_ineq-base_case}  we know that the claim holds for any pair $(n,2)$ with $n \geq 3$. Using this as the basis for our induction, it suffices to show that the claim for the pair $(n+1,k+1)$ follows from the truth of the statement for $(n,k)$.
		Thus, for the induction step assume that
		\begin{equation}\label{eq: complete-tree_induction-hypothesis}
			(n-k)\sum_{e \in L_1^n}x_e+\sum_{e \in L_2^n}x_e \geq k(n-k)\binom{n-k}{2}
		\end{equation}
		is facet-defining for $\Mcut(T_n,S_n)$.
		We show that
		\begin{equation}\label{eq: complete-tree_induction-step}
			((n+1)-(k+1))\sum_{\mathclap{e \in L_1^{n+1}}}x_e+\sum_{\mathclap{e \in L_2^{n+1}}}x_e \geq (k+1) ((n+1)-(k+1))\binom{(n+1)-(k+1)}{2}
		\end{equation}
		is facet-defining for $\Mcut(T_{n+1},S_{n+1})$. 
		
		Let $n+1 \geq 4$ and $k+1 \geq 3$. 
		To prove the induction step, we construct $(n+1)^2$ $S_{n+1}$-multicuts with affinely independent incidence vectors each choosing $k$ or  $k+1$ edges in $L_1^{n+1}$. By \Cref{lem: complete_tree_ineq-valid} these incidence vectors satisfy \eqref{eq: complete-tree_induction-step}.
		
		Since by induction hypothesis \eqref{eq: complete-tree_induction-hypothesis} is facet-defining for $\Mcut(T_n,S_n)$, there exist $S_n$-multicuts $\delta'_1, \dots, \delta'_{n^2}$ in $T_n$ such that $x^{\delta'_1}, \dots , x^{\delta'_{n^2}}$ are affinely independent and attain equality.
		Thus, setting $\delta_i=\delta'_i \cup \{e_{n+1}\}$ for $1 \leq i \leq n$, the set $X^\delta= \{x^{\delta_1}, \dots , x^{\delta_n}\}$ is affinely independent and satisfy \eqref{eq: complete-tree_induction-step} with equality.
		
		By \Cref{lem: linear_independent_subsets} there exist $A_1, \dots, A_{n} \in \binom{[n]}{n-k-1}$ such that the vectors $\sum_{i \in A_1}\mathds{1}_i, \dots,\sum_{i \in A_n}\mathds{1}_i$ are linearly independent.
		Thus, setting 
		\begin{align*}
			\gamma_i &=\{e_j:j \in [n]\setminus A_i\} \cup\{e_{i,j}: 1 \leq j<k\leq n,\ j,k \in A_i\} \cup \{e_{j,n+1}:j \in A_i\} \quad \text{and}\\
			\gamma'_i&=\{e_j:j \in [n]\setminus A_j\} \cup\{e_{i,j}: 1 \leq j<k\leq n,\ j,k \in A_i\} \cup \{f_{j,n+1}:j \in A_i\}
		\end{align*}
		for $1 \leq i \leq n$ the set $X^\gamma=\{x^{\gamma_1},\dots, x^{\gamma_n},x^{\gamma'_1},\dots, x^{\gamma'_n}\}$ is linearly independent. 
		Since $x^{\delta_i}_{e_{j,n+1}}=x^{\delta_i}_{f_{j,n+1}}=0$ for each $1 \leq i \leq n^2$ and $1 \leq j \leq n$, the set $X^\delta \cup X^\gamma$ is affinely independent.
		
		Finally, we set $\delta=\{e_i:1 \leq i\leq k\} \cup \{e_{i,j}: k+1 \leq i <j \leq n+1\}$ and $\mathcal H= \{(n-k-1)x_{e_{n+1}}+\sum_{i \in [n]}x_{e_{i,n+1}}+x_{f_{i,n+1}}=n-k-1\}$.
		Since $X^\delta \cup X^\gamma \subseteq \mathcal H$ and $x^\delta \notin \mathcal H$, the set $X^\delta \cup X^\gamma \cup \{x^\delta\}$ is affinely independent.
		 
		 Since the inequality is supported on $T_n$, the facet is \nice.
	\end{proof}

	Motivated by the NP-hardness of \MinMcut when the input graph is restricted to a tree, we present a polynomial time separation algorithm for generalizations of $(n,k)$-tree inequalities for the multicut dominant in this case.
	We call the inequalities obtained from these facet-defining inequalities by applying \Cref{cor: Replace_edge_by_path} \emph{subdivided $(n,k)$-tree inequalities}. By \Cref{thm: projection} (iv) these yield facet-defining inequalities for each graph containing the according subdivision of $T_n$ with respective sets of terminal pairs.
	As before, we consider our separation algorithm as a proof of concept and are convinced that there are more efficient separation routines utilizing more sophisticated algorithmic approaches whose discussion would be out of scope for this work.
	
	\begin{cor}\label{cor: Spearation_Tree}
		Let $\ell \in \N$ be fixed.
		Given an input graph $G$ that is a tree and a set $S \subseteq \binom{V(G)}{2}$ of terminal pairs, we can enumerate all facet-defining subdivided $(\ell,k)$-tree inequalities for $\Mcut(G,S)$ in polynomial time.
		
		In particular, these inequalities can thus be separated in polynomial time.
	\end{cor}
	\begin{proof}
		We prove that facet-defining subdivided $(\ell,k)$-tree inequalities can be enumerated in polynomial time.
		Checking each enumerated inequality individually yields a simple separation routine.
		
		Let $G=(V,E)$ be a tree with $|V|=n$.
		There are $n \cdot \binom{n-1}{\ell}\cdot \binom{n-1-\ell}{2 \cdot\binom{\ell}{2}}\in \mathcal{O}(n^{2\ell^2 + \ell +1})$ choices for a root $r \in V$, nodes $v_1, \dots , v_\ell \in V \setminus \{r\}$, and nodes $s_{i,j},t_{i,j}\in V \setminus \{r,v_1, \dots, v_\ell\}$, $1 \leq i < j \leq \ell$.
		We check in linear time (in $|V|+|S|$) whether these nodes form the respective nodes of a $T_n$ subdivision in $G$ by searching for the unique $r$-$v_i$-, $v_i$-$s_{i,j}$-, and $v_i$-$t_{i,j}$-paths ($1 \leq i <j \leq \ell$) in $G$ while checking whether all these paths are disjoint and no terminal pair containing nodes which are not in $\{s_{i,j},t_{i,j}:1 \leq i <j \leq \ell\}$ is induced.
		If they do, we can verify whether the leaves induce the necessary terminal pairs in $S$.
		Since $\ell$ is fixed, the number of inequalities corresponding to this tree is constant.
		Hence, we obtain an overall runtime of $\mathcal{O}((|V|+|S|)^{2\ell^2 + \ell +2})$.
	\end{proof}
\section{Cycle Inequalities}\label{sec:cycle-ineq}
	After the investigation of facets of the multicut dominant supported on stars and trees, naturally the question arises whether there are also facet-defining inequalities with $2$-connected support.
	We provide a positive answer to this question by introducing two classes of facet-defining inequalities supported on cycles.
	Throughout this section we consider the cycle $C_n=(\{v_0, \dots , v_{n-1}\}, \{0, \dots , n-1\})$ with $i=v_iv_{(i+1)\bmod n}$ for $0 \leq i <n$.

	First, we consider cycles with each non-edge being a terminal pair:
	\begin{thm}
		Let $n \geq 5$ be odd, consider $C_n$, and let $S=\left\{\{v,w\}:vw \notin E(C_n)\right\}$. Then, the inequality $$\sum_{e \in E(C_n)}x_e \geq \left\lceil \frac{n}{2}\right\rceil $$
		defines a \nice facet of $\Mcut(C_n,S)$.
	\end{thm}
	\begin{proof}
		To prove validity of the inequality let $\delta$ be a multicut in $C_n$. Then, $\delta$  intersects each $2$-path of $C_n$. Since there are $n$ $2$-path and each edge is contained in two such paths, we obtain
		$2 \cdot |\delta| \geq n $ and thus, $|\delta| \geq \left\lceil \frac{n}{2} \right\rceil$.
		
		To show that the inequality is indeed facet-defining, consider the multicuts
		$\delta_i=\{(i+2k)\bmod n: 1 \leq k <\frac{n}{2}\}$ for $1 \leq i \leq n$. Clearly, we have $\sum_{e \in \delta_i}x_e=\left\lceil \frac{n}{2} \right\rceil$. The affine independence of these incidence vectors was shown in the proof of \Cref{thm: circular_star_facet}.
		
		Since the inequality is minimally supported on $C_n$, the facet is \nice.
	\end{proof}
	
	\begin{cor}
		Let $n \geq 5$, consider $C_n$, and let $S=\left\{\{v,w\}:vw \notin E(C_n)\right\} $. Then, $\Mcut(C_n,S)$ is completely described by the inequalities
		\begin{alignat*}{2}
			x_e & \geq 0 \hspace{1cm}&&\text{for all } e \in E(C_n),\\
			x_{uv}+x_{vw} &\geq 1 &&\text{for all } uv,vw \in E(C_n),\\
			\sum_{e \in E(C_n)}x_e &\geq \left\lceil \frac{n}{2}\right\rceil ,&& 
		\end{alignat*}
		where the last inequality can be omitted if and only if $n$ is odd.
	\end{cor}
	\begin{proof}
		We can reuse the proof of \Cref{thm: complete_description_circular-claw} since the polyhedra coincide despite arising from different instances.
	\end{proof}
	
	Finally, we consider another, less dense set of terminal pairs over a cycle.
	There, the graph and terminals form a Moebius ladder instead of a complete graph, cf. \Cref{fig: WagnerIneq}.
	
	\begin{thm}\label{thm: WagnerIneq}
		Let $n \geq 5$ be odd, consider $C_{2n}$, and let $S=\{\{v_i,v_{i+n}\}: 1 \leq i \leq n\}$.
		Then, for $\beta \in \{1,2\}$ and $\beta'=3-\beta$, the inequalities
		$$\sum_{i=0}^{n-1}(\beta x_{2i-1}+\beta' x_{2i}) \geq 3 $$
		define \nice facets of $\Mcut(C_{2n},S)$.
	\end{thm}
	\begin{proof}
		We prove that $\sum_{i=1}^{n}(x_{2i-1}+2 x_{2i}) \geq 3$ defines a shared facet of $\Mcut(C_{2n},S)$. Then, the claim follows for $\beta=2$ by symmetry.
		
		Validity follows from the fact tat the only feasible multicuts with less than three edges are $\delta=\{i,i+n\}$ for $0 \leq i <n$.
		
		It remains to prove that the inequality is indeed facet-defining. To this end consider the multicuts
		\begin{alignat*}{2}
			\delta_i &=	\{i,i+n\} &&\text{for } 0 \leq 1 <n,\\
			\gamma_i &= \left\{(2(i+\ell)) \bmod n: \ell\in\left\{0,1,\left\lceil \frac n 2 \right\rceil \right\} \right\} 
			\quad &&\text{for } 0 \leq 1 <n.
		\end{alignat*}
		Clearly, the incidence vectors of each such multicuts satisfies the inequality under consideration with equality.
		Hence, it remains to prove that these multicuts are affinely independent.
		This is trivial for $x^{\delta_1}, \dots , x^{\delta_n}$.
		Moreover, since $x^{\gamma_i}_j=0$ for each $i$ and each odd $j$ it suffices to prove that $x^{\gamma_1}, \dots, x^{\gamma_n}$ are affinely independent.
		To this end we consider the matrix $A=(a_{i,j})_{0 \leq i,j \leq n}$ with entries $a_{i,j}=x^{\gamma_j}_{2i}$ and show that $A$ has full rank.
	
		Assume $A$ has rank less than $n$. Since $A$ is a circulant matrix, \cite[Theorem~9]{CirculantMatrix_Determinant} together with $\det A=\det A^\mathsf{T}$ yields
	 	$$0= \det A= \prod_{j=0}^{n-1} \left(1+ \zeta^{(n-2)j}+\zeta^{\frac{n-1}{2}j}\right)$$
		where $\zeta=\exp (\frac{2 \mathbf{i} \pi}{n})$.
		Then, there exists some $0 \leq j \leq n$ with 
		$1+ \zeta^{(n-2)j}+\zeta^{\frac{n-1}{2}j}=0$.
		Now, we have $\{\zeta^{(n-2)j}, \zeta^{\frac{n-1}{2}j} \}=\{\exp(\frac{2 \mathbf{i} \pi}{3}), \exp(\frac{4\mathbf{i} \pi}{3}) \}$.
		We have $\zeta^{(n-2)j}=(\zeta^{\frac{n-1}{2}j})^2 = \zeta^{(n-1)j}$ and thus, $j \bmod n =0$ contradicting $1+ \zeta^{(n-2)j}+\zeta^{\frac{n-1}{2}j}=0$.
		
		Since the inequality minimally is supported on $C_{2n}$, the defined facet is \nice.
	\end{proof}
	
	\begin{figure}
		\centering
		\begin{tikzpicture}
			\node[circle,fill, scale=0.5] (1) at	(360/10 * 1 : 1.8cm) {};
			\node[circle,fill, scale=0.5] (2) at	(360/10 * 2 : 1.8cm) {};
			\node[circle,fill, scale=0.5] (3) at	(360/10 * 3 : 1.8cm) {};
			\node[circle,fill, scale=0.5] (4) at	(360/10 * 4 : 1.8cm) {};
			\node[circle,fill, scale=0.5] (5) at	(360/10 * 5 : 1.8cm) {};
			\node[circle,fill, scale=0.5] (6) at	(360/10 * 6 : 1.8cm) {};
			\node[circle,fill, scale=0.5] (7) at	(360/10 * 7 : 1.8cm) {};
			\node[circle,fill, scale=0.5] (8) at	(360/10 * 8 : 1.8cm) {};
			\node[circle,fill, scale=0.5] (9) at	(360/10 * 9 : 1.8cm) {};
			\node[circle,fill, scale=0.5] (10) at	(360/10 * 10 : 1.8cm) {};
			\draw[red,line width=.7mm] (1)--(2);
			\draw[red,line width=.7mm] (3)--(4);
			\draw[red,line width=.7mm] (5)--(6);
			\draw[red,line width=.7mm] (7)--(8);
			\draw[red,line width=.7mm] (9)--(10);
			
			\draw(2)--node[above]{\footnotesize{$\ell=\left\lceil\frac n 2 \right\rceil$}}(3);
			\draw (4)--(5);
			\draw (6)--node[below left]{\footnotesize{$\ell=1$}}(7);
			\draw (8) --node[below right]{\footnotesize{$\ell=0$}}(9);
			\draw (10)--(1);
			\draw[dashed,blue!50] (1)--(6);
			\draw[dashed,blue!50] (2)--(7);
			\draw[dashed,blue!50] (3)--(8);
			\draw[dashed,blue!50] (4)--(9);
			\draw[dashed,blue!50] (5)--(10);
			
			\node () at (0,-2.5) {};
			
			\useasboundingbox (-2,-2.5) rectangle (2cm,2.5cm);
		\end{tikzpicture}
		\caption{The graph from \Cref{thm: WagnerIneq} for $n=5$. Dashed blue connections visualize the set $S$, black edges are those with coefficient $1$, and thick, red edges are those with coefficient $2$. We label the edges of a potential $\gamma_i$-cut}
		\label{fig: WagnerIneq}
	\end{figure}
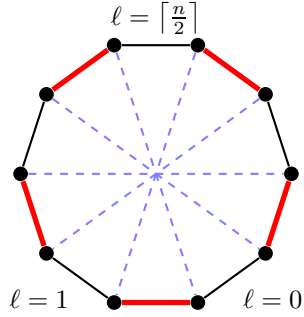
	
	\Cref{thm: WagnerIneq} does not have a natural variant for even $n$.
	However, we can generalize its inequalities by node splittings in the following way:
	\begin{thm}\label{thm: generalized_WagnerIneq}
		Let $n \geq 5$ be odd, $N\geq n$, and $0 < \ell_1 < \dots\ell_{n-1} < \ell_n = N$.
		Consider the cycle $C_{2N}$ and $S_{2N}=\{\{v_i,v_{i+N}\}: 1 \leq i \leq N\}$.
		Then, for $\beta \in \{1,2\}$ and $\beta'=3-\beta$,
		$$\sum_{i=0}^{\ell_1-1}(\beta x_i+ \beta' x_{i+N})+ \sum_{i=\ell_1}^{\ell_2-1}(\beta' x_i+ \beta x_{i+N})+ \dots + \sum_{\mathclap{i=\ell_{n-1}}}^{\ell_n-1}(\beta x_i+ \beta' x_{i+N}) \geq 3 $$
		define \nice facets of $\Mcut(C_{2N},S_{2N})$.
	\end{thm}
	\begin{proof}
		We prove by induction on $N-n$ that the inequality
		is indeed facet-defining for $\beta =1$. Since the inequality is minimally supported on $C_{2n}$, the defined facet is then \nice.
		The claim follows symmetrically for $\beta=2$.
		
		For $N=n$, the claim is given by \Cref{thm: WagnerIneq}. Now, let $N>n$. 
		By induction and symmetry we may assume that
		$$\sum_{i=0}^{\ell_1-1}(x_i+2x_{i+N})+ \sum_{i=\ell_1}^{\ell_2-1}(2x_i+x_{i+N})+ \dots + \sum_{\mathclap{i=\ell_{n-1}}}^{\ell_n-1}(x_i+2x_{i+N}) \geq 3$$
		is facet-defining for $\Mcut(C_{2(N-1)},S_{2(N-1)})$.
		We now construct the claimed facet-defining inequality by using two node splits utilizing \Cref{thm: NodeSplitting}, cf. \Cref{fig: generalized_wagner}.
		First we obtain the graph $G$ from $C_{2(N-1)}$ by splitting $v_{N-1}$ into $v_{N-1}$ and $v_{N-1}'$ such that $v_{N-1}$ is adjacent to $v_N$ and $v_{N-1}'$ is adjacent to $v_{N-2}$, and set $S=S_{2(N-1)}\setminus\{ \{v_0,v_{N-1}\}\} \cup \{\{v_0v_{N-1}\},\{v_0,v_{N-1}'\}\}$.
		A minimum $S$-multicut in $G-v_{N-1}v_{N-1}'$ with respect to the coefficient from the inequality under consideration has value $2$ and is witnessed by $\{0,\ell_1+N\}$ as the former edge separates all terminal pairs but $\{v_0,v_{N-1}\}$, which is separated by the latter edge.
		Hence, \Cref{thm: NodeSplitting} yields that
			\begin{equation}\label{eq: splitted_Wagner-proof}
				1\cdot x_{v_{N-1}v_{N-1}'}+\sum_{i=0}^{\ell_1-1}(x_i+2x_{i+N})+ \sum_{i=\ell_1}^{\ell_2-1}(2x_i+x_{i+N})+ \dots + \sum_{\mathclap{i=\ell_{n-1}}}^{\ell_n-1}(x_i+2x_{i+N}) \geq 3
			\end{equation}
			is facet-defining for $\Mcut(G,S)$.
			
			Now, we obtain $C_{2N}$ from $G$ by splitting $v_0$ into $v_0$ and $v_0'$ such that $v_0$ is adjacent to $v_1$ and $v_0'$ is adjacent to $v_{2N-3}$.
			Within this split we obtain $S_{2N}=S_{2(N-1)}\setminus\{ \{v_0,v_{N-1}\}\} \cup \{\{v_0v_{N-1}\},\{v_0',v_{N-1}'\}\}$.
			Since a minimum $S_{2N}$-multicut in $C_{2N}-v_0v_0'$ with respect to the coefficient from \eqref{eq: splitted_Wagner-proof} is given by $\{v_{N-1}v_{N-1}'\}$ and has value $1$, \Cref{thm: NodeSplitting} yields that
			$$2x_{v_0v_0'}+x_{v_{N-1} v'_{N-1}} +\sum_{i=0}^{\ell_1-1}(x_i+2x_{i+N})+ \sum_{i=\ell_1}^{\ell_2-1}(2x_i+x_{i+N})+ \dots + \sum_{\mathclap{i=\ell_{n-1}}}^{\ell_n-1}(x_i+2x_{i+N}) \geq 3$$
			is facet-defining for $\Mcut(C_{2N},S_{2N})$. After renaming, this is the claimed facet-defining inequality.
	\end{proof}
	\begin{figure}
		\centering
		\begin{subfigure}{.3\textwidth}
			\centering
			\begin{tikzpicture}
				\node[circle,fill, scale=0.5,label={[xshift=3mm,yshift=0mm]\footnotesize $v_{1}$}] (1) at (360/8 * 1 : 1.8cm) {};
				\node[circle,fill, scale=0.5,label={[above,xshift=0,yshift=0mm]\footnotesize $v_{0}$}] (2) at (360/8 * 2 : 1.8cm) {};
				\node[circle,fill, scale=0.5,label={[xshift=-3mm,yshift=0mm]\footnotesize $v_{2(N-1)-1}$}] (3) at (360/8 * 3 : 1.8cm) {};
				\node[circle,fill, scale=0.5] (4) at (360/8 * 4 : 1.8cm) {};
				\node[circle,fill, scale=0.5,label={[below,xshift=-3mm,yshift=-1mm]\footnotesize $v_{N}$}] (5) at (360/8 * 5 : 1.8cm) {};
				\node[circle,fill, scale=0.5,label={[below,xshift=0,yshift=-1mm]\footnotesize $v_{N-1}\phantom{'}$}] (6) at (360/8 * 6 : 1.8cm) {};
				\node[circle,fill, scale=0.5,label={[below,xshift=5mm,yshift=-1mm]\footnotesize $v_{N-2}$}] (7) at (360/8 * 7 : 1.8cm) {};
				\node[circle,fill, scale=0.5] (8) at (360/8 * 8 : 1.8cm) {};
				\draw (1)--(2);
				\draw (2)--(3);
				\draw (3)--(4);
				\draw (4)--(5);
				\draw (5)--(6);
				\draw (6)--(7);
				\draw (7)--(8);
				\draw (8)--(1);
				\draw[dashed,blue!50] (1)--(5);
				\draw[dashed,blue!50] (2)--(6);
				\draw[dashed,blue!50] (3)--(7);
				\draw[dashed,blue!50] (4)--(8);
				
				\node () at (0,-2.5) {};
				
				\useasboundingbox (-2,-2.5) rectangle (2cm,2.5cm);
			\end{tikzpicture}
			\caption{$C_{2(N-1)}$}
		\end{subfigure}
		\hfill%
		\begin{subfigure}{.3\textwidth}
			\centering
			\begin{tikzpicture}
				\node[circle,fill, scale=0.5,label={[xshift=3mm,yshift=0mm]\footnotesize $v_{1}$}] (1) at (360/8 * 1 : 1.8cm) {};
				\node[circle,fill, scale=0.5,label={[above,xshift=0,yshift=0mm]\footnotesize $v_{0}$}] (2) at (360/8 * 2 : 1.8cm) {};
				\node[circle,fill, scale=0.5,label={[xshift=-3mm,yshift=0mm]\footnotesize $v_{2N-3}$}] (3) at (360/8 * 3 : 1.8cm) {};
				\node[circle,fill, scale=0.5] (4) at (360/8 * 4 : 1.8cm) {};
				\node[circle,fill, scale=0.5,label={[below,xshift=-3mm,yshift=-1mm]\footnotesize $v_{N}$}] (5) at (360/8 * 5 : 1.8cm) {};
				
				\node[circle,fill, scale=0.5,label={[below,xshift=0,yshift=-1mm]\footnotesize $v_{N-1}\phantom{'}$}] (6a) at (360/24 * 17 : 1.8cm) {};
				\node[circle,fill, scale=0.5,label={[below,xshift=0,yshift=-1mm]\footnotesize $v_{N-1}'$}] (6) at (360/24 * 19 : 1.8cm) {};

				\node[circle,fill, scale=0.5,label={[below,xshift=5mm,yshift=-1mm]\footnotesize $v_{N-2}$}] (7) at (360/8 * 7 : 1.8cm) {};
				\node[circle,fill, scale=0.5] (8) at (360/8 * 8 : 1.8cm) {};
				\draw (1)--node[above]{\footnotesize$0$}(2);
				\draw (2)--(3);
				\draw (3)--(4);
				\draw (4)--(5);
				\draw (5)--(6a);
				\draw[red] (6a) --(6);
				\draw (6)--(7);
				\draw (7)--(8);
				\draw (8)--(1);
				
				\draw[dashed,blue!50] (1)--(5);
				\draw[dashed,blue!50] (2)--(6);
				\draw[dashed,blue!50] (2)--(6a);
				\draw[dashed,blue!50] (3)--(7);
				\draw[dashed,blue!50] (4)--(8);
				
				\useasboundingbox (-2,-2.5) rectangle (2cm,2.5cm);
			\end{tikzpicture}
			\caption{$G$}
		\end{subfigure}
		\hfill%
		\begin{subfigure}{.3\textwidth}
			\centering
			\begin{tikzpicture}
				\node[circle,fill, scale=0.5,label={[xshift=3mm,yshift=0mm]\footnotesize $v_{1}$}] (1) at (360/8 * 1 : 1.8cm) {};
				
				\node[circle,fill, scale=0.5,label={[above,xshift=0,yshift=0mm]\footnotesize $v_{0}$}] (2a) at (360/24 * 5 : 1.8cm) {};
				\node[circle,fill, scale=0.5,label={[above,xshift=0,yshift=0mm]\footnotesize $v_{0}'$}] (2) at (360/24 * 7 : 1.8cm) {};
				
				\node[circle,fill, scale=0.5,label={[xshift=-3mm,yshift=0mm]\footnotesize $v_{2N-3}$}] (3) at (360/8 * 3 : 1.8cm) {};
				\node[circle,fill, scale=0.5] (4) at (360/8 * 4 : 1.8cm) {};
				\node[circle,fill, scale=0.5,label={[below,xshift=-3mm,yshift=-1mm]\footnotesize $v_{N}$}] (5) at (360/8 * 5 : 1.8cm) {};
				
				\node[circle,fill, scale=0.5,label={[below,xshift=0,yshift=-1mm]\footnotesize $v_{N-1}\phantom{'}$}] (6a) at (360/24 * 17 : 1.8cm) {};
				\node[circle,fill, scale=0.5,label={[below,xshift=0,yshift=-1mm]\footnotesize $v_{N-1}'$}] (6) at (360/24 * 19 : 1.8cm) {};

				\node[circle,fill, scale=0.5,label={[below,xshift=5mm,yshift=-1mm]\footnotesize $v_{N-2}$}] (7) at (360/8 * 7 : 1.8cm) {};
				\node[circle,fill, scale=0.5] (8) at (360/8 * 8 : 1.8cm) {};
				\draw (1)--(2a);
				\draw[red] (2a)--(2);
				\draw (2)--(3);
				\draw (3)--(4);
				\draw (4)--(5);
				\draw (5)--(6a);
				\draw[red] (6a) --(6);
				\draw (6)--(7);
				\draw (7)--(8);
				\draw (8)--(1);
				
				\draw[dashed,blue!50] (1)--(5);
				\draw[dashed,blue!50] (2)--(6);
				\draw[dashed,blue!50] (2a)--(6a);
				\draw[dashed,blue!50] (3)--(7);
				\draw[dashed,blue!50] (4)--(8);
				
				\useasboundingbox (-2,-2.5) rectangle (2cm,2.5cm);
			\end{tikzpicture}
			\caption{$C_{2N}$}
		\end{subfigure}
		\caption{Visualization of the node splitting in the proof of \Cref{thm: generalized_WagnerIneq}. Dashed blue connections indicate terminal pairs. Edges obtained by node splitting are colored red.
		}
		\label{fig: generalized_wagner}
	\end{figure}
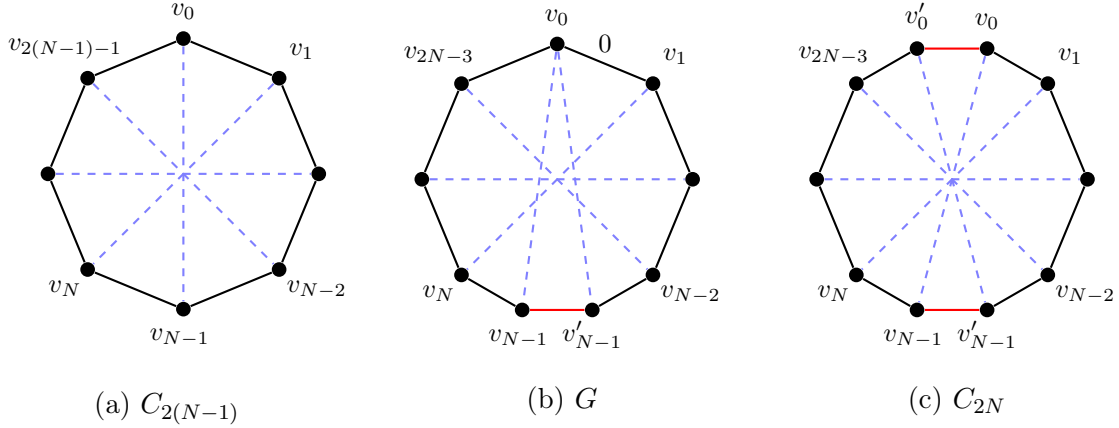
	
	Given the inequalities from Theorems \ref{thm: WagnerIneq} and \ref{thm: generalized_WagnerIneq}, the question arises naturally whether these, together with edge- and path inequalities, suffice to describe the multicut dominant of $C_{2n}$.
	We were able to verify this computationally using normaliz~\cite{normaliz} for $n\leq 7$ and end with the following conjecture:
	
	\begin{conj}\label{conj: DiagonalCycles}
		Let $n \geq 3$, consider $C_{2n}$, and $S= \{\{v_i,v_{i+n}\}: 1 \leq i \leq n\}$. Then, $\Mcut(C_{2n},S)$ is completely defined by the edge inequalities, the path inequalities, and the inequalities from Theorems~\ref{thm: WagnerIneq} and \ref{thm: generalized_WagnerIneq}.
	\end{conj}

\section{Conclusion and Open Problems}
	We introduced the multicut dominant and investigated its basic properties. 
	Moreover, we studied the effect of graph operations on the multicut dominant and its facets.
	We presented facet-defining inequalities supported on stars, trees, and cycles.
	
	Apart from Conjectures \ref{conj: 2component} and \ref{conj: DiagonalCycles}, there are open research questions regarding constraint separation algorithms that we have not tackled in this paper: Are there more efficient separation routines for subdivided star- and tree-equalities on trees than those given in Corollaries \ref{cor: Separation_Stars} and \ref{cor: Spearation_Tree}? Can one give a polynomial separation method for some class of the presented facet-defining inequalities of the multicut dominant of arbitrary graphs?
	
	It is known that \MinMcut is solvable in polynomial time when restricted to $|S| \leq 2$. For $|S|=1$, this is mirrored on the polyhedral side by the complete description of the $s$-$t$-cut dominant.
	All our new facet-defining inequalities require instances with at least $3$ terminal pairs.
	This gives rise to the following conjecture:
	\begin{conj}
		Let $G=(V,E)$ and $S$ be a set of terminal pairs with $|S|=2$. Then $\Mcut(G,S)$ is completely described by the following inequalities:
		\begin{alignat*}{2}
			x_e					&\geq 0,	\qquad 	&& \text{for all } e \in E,\\
			\sum_{e \in (P)} x_e	&\geq 1,			&& \text{for each $s$-$t$-path $P$ with $\{s,t\} \in S. $} 
		\end{alignat*}
	\end{conj}

\bibliographystyle{alpha}
{\footnotesize \bibliography{bibliography}}
\end{document}